\documentclass[11pt]{article}

\usepackage{amsmath,amssymb,amsthm} 
\usepackage[unicode,breaklinks=true,colorlinks=true]{hyperref}
\usepackage[capitalise]{cleveref}
\usepackage[dvipsnames]{xcolor}
\usepackage{mathrsfs}
\usepackage{verbatim}
\usepackage{mathtools}
\usepackage[normalem]{ulem}

\usepackage{tikz}
\usetikzlibrary{decorations.pathreplacing}

\usepackage{soul}
\usepackage{marginnote}

\usepackage[top=1in, bottom=1in, left=1in, right=1in, marginparwidth=1in, marginparsep=0.1in]{geometry}

\numberwithin{equation}{section}

\newtheorem{theorem}{Theorem}[section]
\newtheorem{proposition}[theorem]{Proposition}
\newtheorem{lemma}[theorem]{Lemma}
\newtheorem{definition}[theorem]{Definition} 
\newtheorem{corollary}[theorem]{Corollary}

\theoremstyle{remark}
\newtheorem{remark}[theorem]{Remark}

\definecolor{darkblue}{rgb}{0,0,0.7}
\definecolor{darkred}{rgb}{0.6,0,0}

\newcommand{\al}{\alpha}
\newcommand{\be}{\beta}

\newcommand{\e}{\varepsilon}
\newcommand{\ga}{{\gamma}}
\newcommand{\Ga}{{\Gamma}}
\newcommand{\la}{\lambda}
\newcommand{\La}{\Lambda}
\newcommand{\Om}{{\Omega}}

\newcommand{\si}{\sigma}
\newcommand{\td}{\tilde}

\newcommand{\De}{\Delta}

\newcommand{\Bp}{{\dot B_{p,\I}^{-1+\frac 3 p}}}

\newcommand{\R}{{\mathbb R }}

\newcommand{\N}{{\mathbb N}}
\newcommand{\Z}{{\mathbb Z}}

\newcommand{\cN}{{\mathcal N}}

\newcommand{\pd}{{\partial}}
\newcommand{\nb}{{\nabla}}

\newcommand{\I}{\infty}
 
\renewcommand{\div}{\mathop{\mathrm{div}}}

\newcommand{\supp}{\mathop{\mathrm{supp}}}

\newcommand{\donothing}[1]{{}}

\newcommand{\EQ}[1]{\begin{equation}\begin{split} #1 \end{split}\end{equation}}
\newcommand{\EQN}[1]{\begin{equation*}\begin{split} #1 \end{split}\end{equation*}}

\DeclareMathOperator*{\esssup}{ess\,sup}

\newcommand{\loc}{\mathrm{loc}}

\newcommand{\uloc}{\mathrm{uloc}}

\begin{document}
\title{Spatial decay of discretely self-similar solutions to the Navier-Stokes equations}
\author{Zachary Bradshaw \and Patrick Phelps} 
\date{\today}
\maketitle 
 
\begin{abstract}
Forward self-similar and discretely self-similar weak solutions of the Navier-Stokes equations  are known to exist globally in time for large self-similar and discretely self-similar  initial data and are known to be regular outside of a space-time paraboloid.   In this paper, we establish  spatial decay  rates for such  solutions which hold in the region of regularity provided the initial data has locally sub-critical regularity away from the origin. In particular, we (1) lower the H\"older regularity of the data required to obtain an optimal decay rate for the nonlinear part of the flow compared to the existing literature, (2) establish new decay rates without logarithmic corrections for some smooth data, (3) provide new decay rates for solutions with rough data, and, as an application of our decay rates, (4) provide new upper bounds on how rapidly potentially {non-unique}, scaling invariant local energy solutions can separate away from the origin.
\end{abstract}

\tableofcontents

\section{Introduction}
We investigate the spatial asymptotics of discretely self-similar solutions to the incompressible Navier-Stokes equations in $\R^4_+ = \R^3 \times (0,\I)$. We take viscosity to be unitary, assume forcing to be zero, and denote the
velocity by $u:\R^4_+\to \R^3$ and the pressure by $p:\R^4_+\to \R$. Then, $u$ and $p$ are required to satisfy
\EQ{\label{eq:NS}
\pd_t u -\De u +u\cdot \nb u + \nb p =0, \quad \div u =0,
}
with the initial condition
\EQ{
	u(\cdot, 0) = u_0, \quad \div u_0 = 0,
}
all understood in the sense of distributions.
The Navier-Stokes equations enjoy the following scaling property: If $u$ is a solution, then 
\EQ{u^\la(x,t):= \la u(\la x, \la^2 t), \,\, p^\la(x,t):= \la^2 p(\la x, \la^2 t),}
is also a solution to \eqref{eq:NS}.
A solution is called self-similar (SS) if $u^\la(x,t)=u(x,t)$ for all $\la>0$ and discretely self-similar (DSS) with factor $\la$ if this holds for a given $\la>1$.\footnote{This is a \textit{forward} notion of self-similarity. A \textit{backward} notion is also covered in the literature \cite{leray}.} The data is SS or DSS if the relevant identity holds with the time variable omitted. SS and DSS solutions are interesting as a source of non-uniqueness for \eqref{eq:NS} \cite{JS,JS2,GuillodSverak,ABC} as well as candidates for the failure of eventual regularity of Lemari\'e-Rieusset solutions with data in ultracritical classes {(refer to \cref{sec:mild.sol} for the definition of ultracritical)} \cite{BT1}. The most developed part of the theory of SS and DSS solutions concerns their existence: They are known to exist on $\R^4_+$ for SS or DSS data, respectively, in a variety of function spaces 
\cite{AB,Barraza,BT1,BT2,BT3,BT5,CP,Chae-Wolf,FDLR,GiMi,JS,Kato,KT-SSHS,LR2,Tsai-DSSI}. The most relevant existence results for our work concern large DSS data in $L^{3,\I}$ \cite{BT1}. Note that the preceding list of results excludes existence for DSS initial data in the Koch-Tataru space $\operatorname{BMO}^{-1}$ \cite{KT}---this is an interesting and apparently difficult open problem.

Given the robust existence theory, it is natural to investigate the asymptotic properties of these solutions at large length scales. This is thematically related to proposed work mentioned in the abstract of \cite{GuillodSverak} concerning asymptotic expansions as an important step toward establishing a computer assisted proof of non-uniqueness for the unforced Navier-Stokes equations in the Leray class. See also \cite{KoSv,JS3}.

A preliminary topic is far-field regularity which we presently review. In the self-similar case, Gruji\'c \cite{Grujic} proved that any forward self-similar suitable weak solution is smooth. Indeed, the singular set of a self-similar solution would necessarily be one dimensional in $\R^4_+$, which would violate \cite{CKN}. Gruji\'c's argument breaks down for DSS solutions because their singular sets might possess isolated singularities in space-time, which is not ruled out in \cite{CKN}.  Thus, DSS suitable weak solutions are not known to be smooth and, interestingly, are potential examples of solutions which \emph{do not} exhibit eventual regularity, in contrast to Leray-Hopy weak solutions. Nonetheless, there are examples where smoothness is known. In particular, Kang, Miura and Tsai \cite{KMT} establish smoothness when $u_0\in L^{3,\I}$ is $\la$-DSS and $\la$ is close to $1$. In another direction, Tsai and the first author prove smoothness for DSS local energy solutions evolving from small initial data in $L^2_\uloc$ (these terms are defined in Section 2). These two approaches are unified in \cite{KMT2} where Kang, Miura and Tsai extend their $\la$-close-to-$1$ argument to initial data in a space marginally smaller than $L^2_\uloc$. For $\la$ not close to $1$, the best result on regularity is that any DSS solution in the local energy class with data in $E^2$, which is the closure of the test functions in $L^2_\uloc$, is regular on a set of the form
\EQ{\label{set:FarFieldReg}
\{ (x,t)\in \R^3\times (0,\I) : |x|\geq R_0 \sqrt t  \},
}
where $R_0$ depends on the data. This implies that, for any $\e>0$, $u\in L^\I( \{(x,t) : |x|^2\geq R_0^2t ; t> \e\} )$. A precise statement is recalled below in Theorem \ref{theorem:regularity}. 
Throughout this paper, we use $R_0$ to refer to the region of regularity in the above sense. For globally smooth solutions, $R_0=0$.

We now survey known results on the algebraic decay of DSS solutions. Jia and \v Sver\' ak first constructed SS solutions for large data in \cite{JS}.  Following this, Tsai constructed  \textit{discretely} self-similar solutions for $\la$ close to one, and additionally established the following decay rates using ideas in \cite{JS} under the condition that $u_0\in C^\al_\loc (\R^3\setminus \{0\})$ and $u$ and $u_0$ are $\la$-DSS with scaling factor $\la$ close to $1$:
\EQ{\label{ineq:pointwisedecay1}
| u(x,t)|\lesssim \frac 1 {|x|+\sqrt t};\qquad | \td u(x,t)| \lesssim \frac {\sqrt t} {(|x|+\sqrt t)^2},
} 
where $\td u(x,t)=u(x,t) - e^{t\Delta }u_0(x)$.
Note that, thinking of \eqref{eq:NS} as a perturbation of the heat equation,  $\td u $ can be viewed as the nonlinear part of the evolution of $u$.
If, additionally, $u_0\in C^{1,\al}_\loc (\R^3\setminus \{0\})\cap DSS$ for $\la$ close to $1$, then Tsai showed 
\EQ{\label{ineq:pointwisedecay2}
| \td u (x,t)| \lesssim \frac { t} {(|x|+\sqrt t)^3}\log\bigg(2+\frac {|x|}{\sqrt t}\bigg).
}
This was then improved by Lai, Miao and Zheng to $u_0\in C^{0,1}_\loc(\R^3\setminus \{0\})\cap SS$ in \cite{LMZ} (the reader should refer to the statement of this in \cite[Theorem 1.1]{LMZ2}) and, in \cite{LMZ2}, it is additionally shown that the logarithm can be removed for $u_0\in  C^{1,1}_\loc(\R^3\setminus \{0\})\cap SS$. These results are exclusively for globally smooth solutions (i.e.~SS or DSS with $\la$ close to $1$) while, generally, weak solutions in the DSS class are only  known to be smooth in the region \eqref{set:FarFieldReg}. It is worth emphasizing that the estimates in the foundational literature \cite{JS,Tsai-DSSI} are sufficient for the applications in those papers and were not expected to be optimal---see the discussion in \cite{Tsai-DSSI}.

We mention an older paper of  Brandolese which pioneered this subject for {small}, smooth data \cite{Brandolese}. There,   an asymptotic formula is given for the time-independent profile of a self-similar solution in which the  dominant terms only involve the data. The remaining terms have faster decay, the 
worst of which is $\mathcal O(|x|^{-4})$. This implies spatial asymptotics for the self-similar solution.

Our goals for this paper are to generalize and improve the  decay rates \eqref{ineq:pointwisedecay1} and \eqref{ineq:pointwisedecay2}. In one direction, we establish pointwise bounds for solutions with large, rough initial data, i.e.~$u_0\in L^q_\loc(\R^3\setminus \{0\})$ where $3<q\leq \I$, in the DSS class  for any scaling factor $\la$.  
In another direction, we improve the H\"older regularity required to drop the logarithm in \eqref{ineq:pointwisedecay2} compared to \cite{LMZ2}. We also establish a finer bound on $\td u$ for H\"older regular data compared to \eqref{ineq:pointwisedecay2}.
To summarize these different cases, we will show for $3<q\leq \I$, $0<\al<1$ and $|x|\geq R_0\sqrt t$ that, for DSS data and DSS local energy solutions,
\[|\td u(x,t)|\lesssim
\setlength{\arraycolsep}{0pt}
  \renewcommand{\arraystretch}{1.2}
  \left\{\begin{array}{l @{\quad} l r ll}
 \frac{\sqrt t}{\sqrt t^{\frac{6}{q}} (|x|+\sqrt{t})^{2-\frac{6}{q}}}& u_0\in L^{q}_{\loc}(\R^3\setminus \{0\})  & \quad \text{(identical to \cite{LMZ} when $q=\I$)}  
\\
 \frac { \sqrt t^{1+\al}} {(|x|+\sqrt t)^{2+\al}} 
&u_0 \in C^\al_{\loc}(\R^3\setminus \{0\})  & \text{(new  scale of bounds)}
\\
  \frac { t} {(|x|+\sqrt t)^3}\log\big(2+\frac {|x|}{\sqrt t}\big)& u_0\in C^{1}_{\loc}(\R^3\setminus \{0\})  &   \quad \text{(identical to \cite{LMZ,LMZ2})}
\\
 \frac { t} {(|x|+\sqrt t)^3}& u_0\in C^{1,\al}_{\loc}(\R^3\setminus \{0\})  & \quad \text{(improves  \cite{Tsai-DSSI} and \cite{LMZ2})}
 \end{array}.\right .
\]
We emphasize that, in contrast to the previous works on this subject, all of our results apply to solutions which are only assumed to be regular on \eqref{set:FarFieldReg} and hold for all $\la>1$.  Furthermore, for the cases where we recover the bounds of \cite{LMZ} and \cite{LMZ2}, we do so for a larger class of solutions using simpler methods and without analyzing the Leray equations, which would not work generally for \textit{discretely} self similar solutions. The more complicated method is justified in \cite{LMZ,LMZ2} as they are additionally analyzing the \textit{fractional} Navier-Stokes equations.  Additionally, we explore improvements to these estimates when $e^{t\Delta}u_0$ is replaced by higher Picard iterates in the difference $\td u$.

\bigskip 
We now elaborate on our main results, beginning with our generalization of the decay rates in \cite{JS,Tsai-DSSI} to rough data.

\begin{theorem}[Algebraic decay for rough data]  \label{theorem:Lpdecay1} Let $q\in (3,\I]$ and ${u_0\in L^q_\loc (\R^3\setminus \{0\})}$ be divergence free and DSS. Assume $u$ is a DSS local energy solution with initial data $u_0$.
 It follows that:
\begin{enumerate}
\item For any    $l\in \N_0$ and $|x|\geq R_0\sqrt t$,
\EQ{\label{ineq:pointwiseBound.q}|\nb^l u(x,t)|\lesssim_{u_0,q,\la} \frac{1}{\sqrt{t}^{|l|+\frac{3}{q}} \left( |x|+\sqrt{t}\right) ^{1-\frac{3}{q}}}.} 
\item For the same selection of $x$ and $t$, the difference $\td u := u-e^{t\De}u_0$ satisfies,  
\EQ{\label{ineq:improved.decay.for.difference}
|\td u(x,t)| \lesssim_{u_0,q,\la} \frac{1}{\sqrt t^{\frac{6}{q}-1} (|x|+\sqrt{t})^{2-\frac{6}{q}}}.} 
\end{enumerate}
\end{theorem}

We give some context for the range of exponents $q\in (3,\I]$. In \cite{BT1}, it is shown that $u_0\in L^{3,\I}\cap DSS$ if and only if $u_0\in L^3_\loc(\R^3\setminus \{0\})$. Therefore, the range of exponents achieved in this theorem give decay estimates in  a scale of spaces which approach, but do not reach, the initial data space $L^{3,\I}$ on one end. On the other end, when $q=\I$, this space is close to, but  weaker than, the  H\"older-type spaces considered in  \cite{JS,Tsai-DSSI}. Thus, our estimates fill in the gap between the solutions originally constructed in \cite{JS,Tsai-DSSI} for smooth data and those constructed in \cite{BT1} for rough data. The endpoint case $L^{3,\I}$ from \cite{BT1} \textit{should} be excluded because, as shown in an example in \cite{BT1}, there is no algebraic decay rate for $e^{t\Delta}u_0$ when $u_0\in L^{3,\I}\cap DSS$.

It is reasonable to expect these results are optimal as is explained via a concrete example in Section \ref{sec:optimailty}.

Let $P_0 = P_0 (u_0) = e^{t\Delta}u_0$ and define the $k$-th Picard iterate to be $P_k = P_0 - B(P_{k-1},P_{k-1})$, where 
\[
B(f,g)= \int_0^t e^{(t-s)\Delta}\mathbb P\nb \cdot (f\otimes g)\,ds,
\] and $\mathbb P$ is the Leray projection operator. Note that $ e^{(t-s)\Delta}\mathbb P$ is called the Oseen tensor.

Our next result builds on \cref{theorem:Lpdecay1} by extending the results to the difference of $u$ and higher Picard iterates. \cref{theorem:pwdecay} states that the asymptotics of $u$ are more precisely matched by higher Picard iterates than by $P_0$, but only to a limited extent. Since Picard iterates are unique even for ultracritical data, this leads to an upper bound on how fast two non-unique, DSS local energy solutions can separate locally away from the origin. 

\begin{theorem}[Improved decay using Picard iterates] \label{theorem:pwdecay}
Let $q\in (3,\I]$ and ${u_0\in L^q_\loc (\R^3\setminus \{0\})}$ be divergence free and DSS. Assume $u$ is a DSS local energy solution with initial data $u_0$. 
Define for $k\in \N_0$,
\[
a_k = (k+2)\bigg(1-\frac{3}{q}\bigg) = a_{k-1}+1-\frac 3 q;   \quad k_q=\left\lceil \frac {4q} {q-3}-2  \right\rceil.
\]
   The following hold:
\begin{enumerate}
\item 
We have  for $|x|\geq R_0\sqrt t$ and $k< k_q$ that 
	\EQ{|u-P_k|(x,t) \lesssim_{k,\la,R_0,u_0}\frac{\sqrt{t}^{a_k}}{\sqrt{t}(|x|+\sqrt{t})^{a_k}}.}
\item We have for $|x|\geq R_0\sqrt t$ and $k\geq k_q$ that 
\EQ{\label{ineq:optimalPkdecay}| u-P_{k}|(x,t) \lesssim_{k,\la,R_0,u_0} \frac{\sqrt{t}^3}{(|x|+\sqrt{t})^4}.}
\item   Assume $v$  is another DSS local energy solution with data $u_0$.  Then for  $|x|\geq R_0\sqrt t$, 
\EQ{\label{lim:seperation.scale}
 | u-v |(x,t) \lesssim_{{q},\la,R_0,u_0} \frac{\sqrt{t}^3}{(|x|+\sqrt{t})^4}.  
}

\end{enumerate}
\end{theorem}

To explain the idea behind this theorem, note that, since local energy solutions are known to be mild \cite{LR,BT7}, we have $\td u = B(u,u)$. In other words, by Theorem \ref{theorem:Lpdecay1}, the nonlinear part of the flow enjoys a stronger decay property compared to $u$. Heuristically, owing to the product structure of $B(f,g)$, the decay of $B(f,g)$ is the same as that of $f\cdot \nb g$.
The idea behind \eqref{theorem:pwdecay} is to use the following bi-integral formula which appeared in a related paper of Brandolese \cite{Brandolese}:
\EQ{\label{eq:bi-integral}
	u = \underbrace{P_0	-B(P_0,P_0)}_{P_1})+2B(P_0, \td u) - B( \td u,\td u).
}
Since $P_0$ is $\mathcal{O}(|x|^{-1})$ (see Section \ref{sec:heat}) and $\td u$ is $\mathcal{O}(|x|^{-2})$ (see Theorem \ref{theorem:Lpdecay1}), and based on the above heuristic, for $u_0\in L^{\I}_\loc(\R^3\setminus \{0\})\cap DSS$, as $|x|\to \I$ and $t\in [1,\la^2]$, we should have 
\[
2B(P_0, \td u) = \mathcal O ( |x|^{-3} )\text{ and }B( \td u,\td u) = \mathcal O (|x|^{-4}).
\]
Thus, $u-P_1$ heuristically decays faster than $u-P_0$. We prove this heuristic is merited and the same improvement is evident for higher Picard iterates. This argument is an example of the ``improvement property'' of Picard iterates which has been used widely in the literature, for example in \cite{GIP,Brandolese,AB}.

Item 2 of Theorem \ref{theorem:pwdecay} can be viewed as generalizing  the small data result  of Brandolese \cite[Theorem 2]{Brandolese} in the sense that the solutions in view satisfy, for $t\in [1,\la^2]$, the asymptotic formula
\[
u(x,t) = F(u_0) + \mathcal O (|x|^{-4}),
\]
where $F(u_0)$ can be explicitly computed from the data. Note that \cite{Brandolese} asserts additional structure about the solution which we do not pursue here.

The estimate \eqref{lim:seperation.scale} is interesting as it quantifies the rate at which two \textit{distinct} solutions with the same initial data can separate away from the origin, should two distinct solutions exist. Indeed, while the program of Jia and \v{S}ver\'{a}k \cite{JS,JS2} and the numerics of Guillod and \v{S}ver\'{a}k \cite{GuillodSverak} suggest non-uniqueness within the class of self-similar solutions to the non-forced Navier-Stokes equations,\footnote{Albritton, Bru\'e and Colombo have recently shown non-uniqueness  for related classes for the forced Navier-Stokes equations \cite{ABC}.} not much has been done to quantify how this non-uniqueness would evolve in general scenarios. Part $3$ of \cref{theorem:pwdecay} sheds some light on this. The upper bound arises from the quartic decay of first derivatives of the Oseen tensor.

\bigskip The above asymptotics are for \textit{rough} data, i.e.~data that can be discontinuous away from the origin. We now turn our attention to the case of data which is locally H\"older continuous away from the origin. This is the type of data that is considered in the original papers on the subject \cite{JS,Tsai-DSSI}.  Our theorem fills in the scale of algebraic decay rates between the $L^\I$ case asserted in Theorem \ref{theorem:Lpdecay1} and the $C^{1,\al}$ case in \eqref{ineq:pointwisedecay2}.  When $\al=1$ it recovers the bound obtained in \cite{LMZ}.

\begin{theorem}\label{theorem:tdu.pointwise.decay.CAl}Let $0<\alpha\leq 1$ and assume $u_0\in C^{\alpha}_\loc(\R^3\setminus \{0\})$, is divergence free and is $\la$-DSS. Assume $u$ is a DSS local energy solution with initial data $u_0$.  Then,
\EQ{\label{ineq.HolderDecayWithLog}|\td u| (x,t) \lesssim_{\|u_0\|_{L^2_\uloc}, \|u_0\|_{C^\al(A_0)}}
\begin{cases}
\frac{t^\frac{1+\al}{2}}{(\sqrt t+|x|)^{2+\al}}  & \al<1\\
\frac{t }{(\sqrt t+|x|)^{3}}\log(2+ \frac {|x|}{\sqrt{t}}) & \al=1
\end{cases},
}
for $|x|\geq R_0\sqrt t$.
\end{theorem}
 
In analogy with Theorem \ref{theorem:pwdecay}, one can also pursue improved decay estimates for $u-P_k$ where $k>0$. This improvement terminates in one step once quartic decay is reached.

The fact that the logarithm can be dropped when $\al<1$ is subtle as  we presently explain. 
Observe that 
\[
\td u = B\big(\td u,\td u\big) + B\big(\td u,{P_0}\big) + B\big({P_0},B(u,u-{P_0})+B(u,{P_0}) \big) + B({P_0},{P_0}).
\]
Using pointwise estimates for the Oseen tensor, the first three terms written above all satisfy bounds \textit{stronger} than $(|x|+1)^{-2-\alpha}$ for $t\in [1,\la^2]$. 
Therefore, the term $B({P_0},{P_0})$ will fully determine our estimates for $\td u$. This term has the benefit that we can write it explicitly in terms of $u_0$:
\[
B({P_0},{P_0}) = \int_0^t e^{(t-s)\Delta}\mathbb P \nb\cdot  (  e^{t\Delta}u_0\otimes e^{t\Delta}u_0 )\,ds.
\]
Because $u_0$ is only in $C^\al_\loc(\R^3\setminus\{0\})\cap DSS$, $\nb e^{t\Delta}u_0$ only decays like $|x|^{-1-\al}$. Hence, $ \nb\cdot (  e^{t\Delta}u_0\otimes e^{t\Delta}u_0 )$ decays like $|x|^{-2-\al}$. When we try to extract the same estimate for $B({P_0},{P_0})$, the kernel of the Oseen tensor, which is a cubic power of $|x|^{-1}$, must be integrated and this introduces a logarithm.  However, our analysis of the heat equation in Section \ref{sec:heat} shows    that  $|\La^\be e^{t\Delta}u_0 |$ also decays like $|x|^{-1-\al}$ where $\La=(-\Delta)^{\frac 1 2}$ and $\al<\be<1$. 
After navigating a commutator, this implies
 \[
|\La^{\be} (  e^{t\Delta}u_0\otimes e^{t\Delta}u_0 )| \lesssim |x|^{-2-\al}, \quad t\in [1,\la^2]
\]
Then, we consider a revised formula for $B({P_0},{P_0})$, namely
\[
B({P_0},{P_0}) = \int_0^t \La^{-\be} \nb \mathbb P  e^{(t-s)\Delta}\La^{\be} ( e^{t\Delta}u_0\otimes e^{t\Delta}u_0 )\,ds.
\]
This no longer has a cubic power of $|x|^{-1}$ in the Oseen part (and so does not lead to a logarithm) but still has decay on the order of $|x|^{-2-\al}$ from the product part.

Finally, we improve a result in \cite{LMZ2} which says the optimal   decay rate is achieved without a logarithm as in \eqref{ineq:pointwisedecay2} when the data is in $C^{1,1}_\loc(\R^3\setminus \{0\})\cap SS$. 
\begin{theorem}\label{theorem:tdu.pointwise.decay.C1Al}
Let $0<\alpha\leq 1$ and assume $u_0\in C^{1,\alpha}_\loc(\R^3\setminus \{0\})\cap DSS$ and is divergence free.  Assume $u$ is a DSS local energy solution with initial data $u_0$.  Then,
\EQ{\label{ineq.HolderDecayCubic}
{|\td u |}(x,t) \le\frac{Ct}{(|x|+\sqrt{t})^3},}
for $|x|\geq R_0\sqrt t$.
\end{theorem}

This says that, although a logarithm appears in the $C^{1}_\loc(\R^3\setminus \{0\})$ case in Theorem \ref{theorem:tdu.pointwise.decay.CAl}, it disappears when any higher smoothness is assumed. 
No algebraic gain is achieved as $\al$ increases due to the explicit estimates for small data self-similar solutions with smooth initial data in \cite{Brandolese}.
Our proof of Theorem \ref{theorem:tdu.pointwise.decay.C1Al} is substantially different than  the proof for the $C^{1,1}$ case in \cite{LMZ2} and utilizes the fractional Laplacian in a similar manner to that discussed above.

\section{Preliminaries}\label{sec:solutionclasses}

\subsection{Function spaces}

Here we introduce function spaces which play an important role in this paper or the surrounding subject area. Solutions constructed based on these spaces will then be discussed.

The $L^p$ spaces and $L^p_\loc$ classes are defined in the classical way. 
Uniformly local versions are denoted $L^p_\uloc$ and defined by finiteness of the norm
\[
\| f\|_{L^p_\uloc} :=\sup_{x_0\in \R^3} \| f\|_{L^p(B_1(x_0))}.
\]We denote by $E^p$ the closure of $C_c^\I$ in $L^p_\uloc$. This class is characterized by the condition
\[
\lim_{R\to \I} \| f\|_{L^p_\uloc(\R^3\setminus B_R)}=0.
\]
The endpoint Lorentz spaces $L^{p,\I}$ are defined by finiteness of the quasinorm
\[
\| f\|_{L^{p,\I}} := \sup_{\si>0} \si^p | \{x:\si<|f(x)| \}|. 
\]

Let $\mathcal K_p$ be the Kato class defined by the finiteness of the norm
\[
\| u \|_{\mathcal K_p} :=\esssup_{t>0} t^{\frac 1 2(1-\frac 3 p)} \| u(t)\|_{L^p}.
\] 
The Besov spaces relevant to DSS solutions can be defined using Kato classes.
Assuming $3<p\leq \I$, $f\in \Bp(\R^3)$ if and only if 
\[
\| e^{t\Delta} f\|_{\mathcal K_p} <\I,
\] 
the above norm being equivalent to the norm classically defined using Littlewood-Paley.

The Koch-Tataru space $\operatorname{BMO}^{-1}$ is defined by finiteness of the following norm:
\EQ{\|e^{t\Delta}u_0\|_{\operatorname{BMO}^{-1}} := \esssup_{t\in(0,\I)} t^{\frac{1}{2}}\|e^{t\Delta}u_0\|_{L^\I(\R^3)}+\sup_{x\in\R^3}\sup_{R\in(0,\I)}R^{-\frac{3}{2}}\|e^{t\Delta}u_0\|_{L^2((B(x,R)\times (0,R^2))}.}

We define H\"older classes and spaces as usual. In particular, for $0<\al\leq 1$ we let 
\[
 [f]_{C^{0,\al}(\Om)} =  \sup_{x\neq y\in \Om} \frac {|f(x)-f(y)|}{|x-y|^\al}.
\] 
 We make the abbreviation $C^\al = C^{0,\al}$ when there is no confusion. We say $f\in C^\al_\loc ( \Om)$ if $f\in C^{\al}(\Om')$ for all compact subsets $\Om'$ of $\Om$. Finally, we set
\[
\|f\|_{C^{k,\al}(\Om)} = \max_{|\be|\leq k}\, \sup_{x\in\Om} |D^{\be} f|(x) + \max_{|\be|=k} [D^\be f]_{C^{0,\al}(\Om)}, 
\]
and define $C^{k,\al}_\loc$ in analogy with $C^\al_\loc$.

\subsection{Mild solutions, the Oseen tensor, and Picard iterates}\label{sec:mild.sol} A mild solution is a solution to \eqref{eq:NS} with the form
\[
u(x,t)=e^{t\Delta}u_0 -\int_0^te^{(t-s)\Delta} \mathbb P \nb \cdot (u\otimes u)\,ds =: e^{t\Delta}u_0 - B(u,u),
\]
which is just Duhamel's formula applied to the following version of \eqref{eq:NS}:
\[
\partial_t u -\Delta u = - \mathbb P ( u\cdot \nb u ).
\]
Mild solutions are not expected to be regular in general \cite{LR,BT7}, but in the classical literature they were introduced in the context of strong solutions \cite{FJR,Kato} obtained  as a limit of Picard iterates. An important line of research concerns which function spaces guarantee global well-posedness for small data. Example of spaces where a positive answer is available are 
\[
L^3 \subset L^{3,\I} \subset \Bp (3<p<\I)\subset \operatorname{BMO}^{-1}.
\] 
With the exception of $L^3$, these spaces are \emph{ultracritical} in that the closure of the test functions is not dense in them. Ultracritical spaces include self-similar and DSS data. 

All of the self-similar and DSS solutions constructed in \cite{JS,Tsai-DSSI,LR2,BT1,FDLR,AB} are mild due to sufficient conditions in \cite{LR,BT7}. In particular, for our work, the local pressure expansion in the definition of local energy solutions (defined in the next section) is equivalent to being mild and therefore we are free to use mild solution estimates in our analysis.

We will need some details on the structure of the Oseen tensor which we presently recall, drawing on the sources \cite{LR,BV,MaTe,Tsai-book,LMZ2,VAS}. 
\begin{lemma}[Bound for derivatives of the Oseen tensor]\label{lem:oseen.structure}
The operator $e^{t\De}\mathbb{P} \nb \cdot $ (where $\mathbb{P}$ is the Helmholtz projection) in $\R^3$ with kernel $\nb S(x,t)$ has the following bound:
\[|D^l_x \pd_t^m S_{j,k}(x,t)| \le C (t^\frac 1 {2} + |x|)^{-3-l-2m}, \, \forall l,m \in \Z^+.\]
\end{lemma}
In \cite{DL}, pointwise bounds for the Oseen tensor involving fractional powers of the Laplacian are established and we recall the details modified slightly to match our notation.
\begin{proposition}[\cite{DL}, Proposition 3.1]\label{prop:DL.Oseen.est} For any integer $m\ge 0$ and $-1<\al\le 1$, 
	\EQ{\left|(|x|+1)^{3+m+\al}D^m\La^\al S_{j,k}\right|(x,1) \lesssim_{m,\al} 1,\qquad \forall x\in \R^d.}
\end{proposition}
We use this estimate in the following form
	\EQ{\left|\nb\La^{-\be} S_{j,k}\right|(x,t) \lesssim_{\be} (|x|+\sqrt{t})^{-4+\be},}
where $0<\be<1$.

We finally examine the boundedness of Picard iterates in several contexts.
By \cite[(1.8)-(1.10)]{MaTe}, if $u_0\in L^{3,\I}$, all Picard iterates are in $L^2_\uloc$ for all finite times, but bounds degrade at higher iterates. In particular, from \cite[(1.8)]{MaTe},
\[
\| e^{t\Delta}u_0\|_{L^2_\uloc} \lesssim \|u_0\|_{L^2_\uloc} \lesssim \|u_0\|_{L^{3,\I}}.
\]
Recalling that $P_0=e^{t\Delta}u_0$, it follows from \cite[(1.10)]{MaTe} that 
\EQ{\label{ineq:PicardIteratesL2uloc}
\| P_1-P_0\|_{L^2_\uloc }(t) &\leq \int_0^t \frac 1 {\sqrt{t-s} \sqrt s} \| P_0\|_{\mathcal K_\I} \| P_0\|_{L^2_\uloc}\,ds \lesssim \|u_0\|_{L^{3,\I}}^2,
}
for all times, where we are using $\|P_0\|_{\mathcal K_\I} \sim \| u_0\|_{\dot B^{-1}_{\I,\I}}\lesssim \|u_0\|_{L^{3,\I}}$ by continuous embeddings. This can be extended to higher Picard iterates noting that all Picard iterates satisfy 
\[
\sup_{t>0}\| P_k\|_{L^\I(\R^3)}(t) t^{\frac 1 2 } = \|P_k\|_{\mathcal K_\I}\lesssim_k \|u_0\|_{L^{3,\I}}.
\]
This well-known fact can be checked inductively by keeping track of the above quantity, as well as $\|P_k\|_{\mathcal K_4}$, using the following estimates:
\EQ{
\|P_{k+1}\|_{L^4}&\lesssim \int_0^t \frac 1 {(t-s)^{\frac 1 2}} \|P_k^2\|_{L^4}\,ds
\\&\lesssim \int_0^t \frac 1 {(t-s)^{\frac 1 2} s^{\frac 1 2+\frac 1 8}} \|P_k\|_{\mathcal K_\I}\|P_k\|_{\mathcal K_4}\,ds
\\&\lesssim t^{-\frac 1 8}\|P_k\|_{\mathcal K_\I}\|P_k\|_{\mathcal K_4},
}
and
\EQ{
\|P_{k+1}\|_{L^\I}&\lesssim \int_0^t \frac 1 {(t-s)^{\frac 1 2 +\frac 3 8}} \|P_k^2\|_{L^4}\,ds
\\&\lesssim \int_0^t \frac 1 {(t-s)^{\frac 78} s^{\frac 1 2+\frac 1 8}} \|P_k\|_{\mathcal K_\I}\|P_k\|_{\mathcal K_4}\,ds
\\&\lesssim t^{-\frac 1 2}\|P_k\|_{\mathcal K_\I}\|P_k\|_{\mathcal K_4}.
}

\subsection{Local energy solutions}\label{sec:LE.solutions}

In this subsection, we define local energy solutions and compile some known properties that will be needed in what follows. These solutions were introduced by Lemari\'e-Rieusset \cite{LR} and played an important role in the proof of local smoothing in \cite{JS}. Because $L^{3,\I}\subset L^2_\uloc$ it is a natural class in which to consider SS and DSS solutions \cite{JS,BT1}.

\begin{definition}[Local energy solutions]\label{def:localEnergy} A vector field $u\in L^2_{\loc}(\R^3\times [0,T))$, $0<T\leq \I$, is a local energy solution to \eqref{eq:NS} with divergence free initial data $u_0\in L^2_{\uloc}(\R^3)$, denoted as $u \in \cN(u_0)$, if:
\begin{enumerate}
\item for some $p\in L^{\frac{3}{2}}_{\loc}(\R^3\times [0,T))$, the pair $(u,p)$ is a distributional solution to \eqref{eq:NS},
\item for any $R>0$, $u$ satisfies
\begin{equation}\notag
\esssup_{0\leq t<R^2\wedge T}\,\sup_{x_0\in \R^3}\, \int_{B_R(x_0 )}\frac 1 2 |u(x,t)|^2\,dx + \sup_{x_0\in \R^3}\int_0^{R^2\wedge T}\int_{B_R(x_0)} |\nb u(x,t)|^2\,dx \,dt<\I,\end{equation}
\item for any $R>0$, $x_0\in \R^3$, and $0<T'< T $, there exists a function of time $c_{x_0,R}\in L^{\frac{3}{2}}_{T'}$\footnote{The constant $c_{x_0,R}(t)$ can depend on $T'$ in principle. This does not matter in practice and we omit this dependence.} so that, for every $0<t<T'$  and $x \in B_{2R}(x_0)$  
\EQ{ \label{eq:pressure.dec}
p(x,t)&= c_{x_0,R}(t)-\De^{-1}\div \div [(u\otimes u )\chi_{4R} (x-x_0)]
\\&\quad - \int_{\R^3} (K(x-y) - K(x_0 -y)) (u\otimes u)(y,t)(1-\chi_{4R}(y-x_0))\,dy 
,
}
in $L^{\frac{3}{2}}(B_{2R}(x_0)\times (0,T'))$
where  $K(x)$ is the kernel of $\De^{-1}\div \div$,
 $K_{ij}(x) = \pd_i \pd_j \frac {-1}{4\pi|x|}$, and $\chi_{4R} (x)$ is the characteristic function for $B_{4R}$. 
\item for all compact subsets $K$ of $\R^3$,  $u(t)\to u_0$ in $L^2(K)$ as $t\to 0^+$,
\item $u$ is suitable in the sense of Caffarelli-Kohn-Nirenberg, i.e., for all cylinders $Q\Subset Q_T$ and all non-negative $\phi\in C_c^\I (Q)$, we have  the \emph{local energy inequality}
\EQ{\label{ineq:CKN-LEI}
2\iint |\nb u|^2\phi\,dx\,dt \leq \iint |u|^2(\pd_t \phi + \De\phi )\,dx\,dt +\iint (|u|^2+2p)(u\cdot \nb\phi)\,dx\,dt,
}
\item the function
\EQ{
t\mapsto \int_{\R^3} u(x,t)\cdot {w(x)}\,dx,
}
is continuous in $t\in [0,T)$, for any compactly supported $w\in L^2(\R^3)$.
\end{enumerate}
\end{definition}

Local energy solutions are known to satisfy certain \textit{a priori} bounds \cite{LR}. For example,in \cite{JS}, the following a priori bound is proven: 
Let $u_0\in E^2$, $\div u_0=0$, and assume $u\in \mathcal N (u_0)$.  For all $r>0$ we have
\begin{equation}\label{ineq.apriorilocal}
\esssup_{0\leq t \leq \sigma r^2}\sup_{x_0\in \R^3} \int_{B_r(x_0)}\frac {|u|^2} 2 \,dx\,dt + \sup_{x_0\in \R^3}\int_0^{\sigma r^2}\int_{B_r(x_0)} |\nabla u|^2\,dx\,dt <CA_0(r) ,
\end{equation}
where
\[
A_0(r)=rN^0_r= \sup_{x_0\in \R^3} \int_{B_r(x_0)} |u_0|^2 \,dx,
\] 
and
\begin{equation}\label{def.sigma}
\si=\sigma(r) =c_0\, \min\big\{(N^0_r)^{-2} , 1  \big\},
\end{equation}
for a small universal constant $c_0>0$.

The following theorem is a re-statement of results in \cite{KMT2}. It serves is the mathematical foundation for the parameter $R_0$ used throughout this paper. 

\begin{theorem}[Far-field regularity]\label{theorem:regularity}
Fix $\la>1$.
Let $u$ be a $\la$-DSS local energy solution of the Navier-Stokes equations in $\R^3\times(0,T)$ with divergence free, $\la$-DSS initial data $u_0\in E^2$. 
Then, there exists $R_0 = R_0 (u_0)$ so that  
$u$ is smooth  and bounded on  
\[
\{  (x,t): |x| \geq  R_0\sqrt  t  ; 1\leq t\leq \la^2   \}.
\]  
This implies smoothness on $\{  (x,t): |x|\geq  R_0 \sqrt t      \}$.
Furthermore, there exists $\la_0(u_0)>1$ so that if $1\leq \la \leq \la_0$ then $u$ is globally smooth.
\end{theorem}

Local energy solutions were introduced with the partial regularity of Caffarelli, Kohn and Nirenberg \cite{CKN} in mind. In particular, this and the subsequent variations of their partial regularity can be used when analyzing local energy solutions. 
We will use the following $\varepsilon$-regularity criteria of Lin \cite{L98}.
\begin{lemma}[$\e$-Regularity \cite{L98}] \label{lemma:ereg} 
There exists a universal small constant $\varepsilon_*>0$ such that, if the pair $(u,p)$ is a suitable weak solution of (NS) in $Q_r = Q_r(x_0,t_0)=B_r(x_0)\times (t_0-r^2, t_0),\, B_r(x_0)\subset \R^3$, and
\EQ{\varepsilon^3 = \frac{1}{r^2} \int_{Q_r} (|u|^3 +|p|^{\frac{3}{2}}) dxdt <\varepsilon_*,}
then $u\in L^\I(Q_{\frac{r}{2}})$. Moreover, 
\EQ{\|\nb^k u \|_{L^\I(Q_{\frac{r}{2}})} \le C_k \varepsilon r^{-k-1},}
for universal constants $C_k$, where $k\in \N_0$.
\end{lemma}

\begin{remark}\label{rem1}If $(u,p)$ is a suitable weak solution to (NS), then $(u,\td{p})$, where $\td{p} = p(x,t)-c_{x_0,R}(t)$ and $c_{x_0,R}(t)\in L^{\frac{3}{2}}(0,T)$, is a suitable solution to (NS)as follows: the hypothesis on $u$ are unchanged, $\nb p = \nb \td{p}$ in $\mathcal D'$, and $(u,\td{p})$ still satisfies the local energy inequality since
	\EQ{\int_0^t \int_\Om \td{p} u\cdot \nb \phi dxdt&=\int_0^t \int_\Om p u\cdot \nb \phi dxdt-\underbrace{\int_0^t\int_\Om c_{x_0,R}(t)u\cdot \nb \phi dxdt}_{=0},
	}
due to the divergence free condition.
\end{remark}

Another important property of local energy solutions concerns local smoothing as developed by Jia and \v Sver\' ak in \cite{JS}.

\begin{theorem}[Local smoothing {\cite[Theorem 3.1]{JS}}]
	\label{theorem:JSlocalsmoothing} Let $u_0\in L^2_{\uloc}(\R^3)$ be divergence free.
	Suppose $u_0\in L^q(B_2(0))$ with $\|u_0\|_{L^q(B_2(0))}\le M<\I$ and $q>3$. Decompose $u_0 = u_0^1+u_0^2$ with $\div u_0^1 =0,\,u_0^1|_{B_{4/3}} = u_0,\, \operatorname{supp} u_0^1 \Subset B_2(0)$, and $\|u_0^1\|_{L^q(\R^3)}<C(M,q)$. Let $a$ be the locally in time defined mild solution to Navier-stokes equations with initial data $u_0^1$. Then there exists a positive $T=T(\alpha,q,M)>0$, such that any local energy solution $u\in \mathcal{N}(u_0)$ satisfies 
		\EQ{\|u-a\|_{C^\gamma_{par}(\overline B_{\frac{1}{2}}\times [0,T])}\le C(M,q,\|u_0\|_{L^2_\uloc}),} for some $\gamma=\gamma(q)\in(0,1)$.
\end{theorem}
Note that local smoothing directly implies far-field regularity in the classes we study in this paper.

\subsection{Pointwise estimates for convolution operators}

We recall pointwise bounds for some convolutions from \cite[Lemma 2.1]{Tsai-DSSI}: Let $0<a<5,\, 0<b<5$ and $a+b>3$. Then, 
\EQ{\phi(x,a,b) = \int_0^1 \int_{\R^3} (|x-y|+\sqrt{1-t})^{-a}(|y|+\sqrt{t})^{-b} \,dy\,dt,} 
is well defined for $x\in \R^3$, and
\EQ{\label{ineq:Tsai.integral}\phi(x,a,b) \lesssim R^{-a} + R^{-b} + R^{3-a-b} [ 1+ (1_{a=3}+1_{b=3})\log R],}
where $R=|x|+2$. 
 These estimates can be extended to  {other time intervals} by a change of variable. We also need a variation of this estimate, which is the content of the next lemma.   

\begin{lemma}\label{lemma.intbound1}For $a\in [0,5)$ and $b\in[0,2)$, where $a+b<5$, 
	\EQ{\int_0^t\int \frac {1} {(|x-y| +\sqrt{t-s})^4} \frac 1 {(|y|+\sqrt s)^a} \frac 1 {\sqrt s^b} \,dy\,ds \leq \frac C {\sqrt t^{1-a}(|x|+\sqrt t)^a}+\frac C {\sqrt t^{1-4}(|x|+\sqrt t)^4}.
	}
\end{lemma}
\begin{proof}
We begin with the case $t=1$.
Let $R=|x|+2$. It is easy to see that 
	\EQ{\sup_{R\leq 8}\int_0^1\int \frac {1} {(|x-y| +\sqrt{1-s})^4} \frac 1 {(|y|+\sqrt s)^a} \frac 1 {\sqrt s^b} \,dy\,ds <\I,}
provided $a+b<5$. Hence we just need to verify the correct decay when $R>8$. First consider the far-field, $|y|>2R$, where
	\EQ{\label{ineg:far.field.int.est}\int_0^1 \int_{|y|>2R}\frac {1} {(|x-y| +\sqrt{1-s})^4} \frac 1 {(|y|+\sqrt s)^a} \frac 1 {\sqrt s^b} \,dy\,ds &\lesssim \int_0^1 \frac{1}{\sqrt{s}^b} \,ds \int_{|y|>2R} \frac{1}{|y|^{4+a}}\,dy \\&\lesssim \frac{1}{R^{1+a}}.
	}
Next consider the region where $|y|<R/2$. Because $R>8$, for $a\ne 3$ ,
	\EQ{\label{ineq:near.field.int.est}\int_0^1 \int_{|y|<\frac{R}{2}}&\frac { 1} {(|x-y| +\sqrt{1-s})^4} \frac 1 {(|y|+\sqrt s)^a} \frac 1 {\sqrt s^b} \,dy\,ds\\ &\lesssim \int_0^1 \int_{|y|<\frac{R}{2}}\frac { 1} {R^4} \frac 1 {(|y|+\sqrt s)^a} \frac 1 {\sqrt s^b} \,dy\,ds 
		\\&\lesssim \frac { 1} {R^4} \int_0^1 \frac 1 {\sqrt s^b }\int_{|y|<\frac{R}{2}} \frac 1 {(|y|+\sqrt{s})^a} \,dy\,ds.}
Passing to spherical coordinates yields
	\EQ{
		\frac { 1} {R^4} \int_0^1 \frac 1 {\sqrt s^b }\int_{|y|<\frac{R}{2}} \frac 1 {(|y|+\sqrt{s})^a} \,dy\,ds
		&\lesssim \frac { 1} {R^4} \int_0^1 \frac 1 {\sqrt s ^b} \int_0^{\frac{R}{2}} (r+\sqrt{s})^{-a}r^2\, dr\, ds
		\\&\lesssim \frac { 1} {R^4} \int_0^1 \frac 1 {\sqrt s ^b } \left(\bigg(\frac{R}{2}+\sqrt{s}\bigg)^{3-a}-\sqrt{s}^{3-a}\right) \,ds\\& \lesssim \left(\frac{1}{R^{1+a}}+\frac{1}{R^4}\right)\int_0^1 \frac 1 {\sqrt s ^{a+b-3 }}\,ds
		\\& \lesssim \frac{1}{R^{1+a}}+\frac{1}{R^4},
	} 
because $a+b<5$. If $a=3$, then, similarly, passing to spherical coordinates yields
	\EQ{\label{ineq:near.field.int.est.a=3}\int_0^1 \int_{|y|<\frac{R}{2}}\frac { 1} {(|x-y| +\sqrt{1-s})^4} \frac 1 {(|y|+\sqrt s)^3} \frac 1 {\sqrt s^b} \,dy\,ds
		&\lesssim \frac{\ln(R+1)}{R^4}+\frac 1 {R^4}\lesssim \frac 1 {R^a} +\frac 1 {R^4},}
as long as $b<2$ since $a=3<4$.
	
The final region, $R/2<|y|<2R$, is treated as follows: 
	\EQ{\label{ineq:mid.field.int.est}\int_0^1 \int_{R/2<|y|<2R}&\frac { 1} {(|x-y| +\sqrt{1-s})^4} \frac 1 {(|y|+\sqrt{s})^a} \frac 1 {\sqrt s^b} \,dy\,ds\\ &\lesssim \int_0^1 \int_{R/2<|y|<2R}\frac { 1} {(|x-y| +\sqrt{1-s})^4} \frac 1 {R^a} \frac 1 {\sqrt s^b} \,dy\,ds\\ &\lesssim \frac{1}{R^a}\int_0^1 \frac{1}{\sqrt{s}^b} \int_{|z|<3R} \frac{1}{(|z| +\sqrt{1-s})^4}\,dz \,ds \\
		& \lesssim \frac 1 {R^a}\int_0^1 \frac{1}{\sqrt{s}^b} \left( -\frac{1}{3R+\sqrt{1-s}} + \frac{1}{\sqrt{1-s}}\right) \, ds \\
		&\lesssim \frac 1 {R^a} + \frac 1 {R^{a+1}},
	}
where we passed to spherical coordinates to evaluate the spatial integral.
	
Therefore, for all $x$,
	\EQ{\int_0^1\int \frac {1} {(|x-y| +\sqrt{1-s})^4} \frac 1 {(|y|+\sqrt s)^a} \frac 1 {\sqrt s^b} \,dy\,ds \le \frac{C}{(|x|+1)^a}+\frac{C}{(|x|+1)^4}, }
To conclude the proof we make the change of variables $x=\sqrt t \td x$, $y=\sqrt t \td y$ and $s= t \td s$.\ This substitution and the above bound lead to 
	\EQ{\label{ineq:int.bound.CoV}
	\int_0^t\int \frac {1} {(|x-y| +\sqrt{t-s})^4} \frac 1 {(|y|+\sqrt s)^a} \frac 1 {\sqrt s ^b} \,dy\,ds 
	& \leq \frac 1 {\sqrt t} \left( \frac C {(|\td x| +2)^a} + \frac C {(|\td x|+2 )^4}\right)\\
	&\leq     \frac C {\sqrt t^{1-a}(|x|+\sqrt t)^a}+ \frac {Ct^3} {(|x|+\sqrt t)^4}.
	} 
\end{proof}

\section{Analysis of the heat equation}\label{sec:heat}
In this section, we state and prove a variety of results on the decay of DSS solutions to the heat equation. These will be foundational for our work on Navier-Stokes. Let $A_k=\{ x\in \R^3: \la^k\leq |x|<\la^{k+1} \}$ and $A_k^* = \{x: \la^{k-1}\leq |x|<\la^{k+2}\}$.
\begin{lemma}\label{lemma:heat.equation.pointwise.decay}Assume $f\in L^q_\loc(\R^3\setminus \{0\})$, where $3<q\leq \I$, and satisfies, for some $\la >1$, 
	\[
	\la^{\sigma k} f(\la^k x) = f(x),
	\] 
where $\sigma\in \big(\frac 3 q, \infty\big)$ (note that $\sigma=1$ corresponds to being DSS). 
Then,
	\[
	\sup_{t\in [1,\la^2]} \| e^{t\Delta} f \|_{L^\I(B_R^c)} \lesssim_\la \| f\|_{L^q(A_1)} R^{ \frac 3 q -\sigma }.
	\]
\end{lemma}
Note that this conclusion is discussed without proof by Tsai and the first author in \cite{BT1}. When $q=3$, there is no algebraic decay rate available as demonstrated in \cite{BT1}.
\begin{proof}
It suffices to prove that, for $t\in [1,\la^2]$, 
	\[
	\| e^{t\Delta}f \|_{L^\I (A_k)}\lesssim \la^{k(\frac 3 q -\sigma)},
	\]
where the suppressed constants are independent of $k$. For any $x$ and $k$, we will consider
	\EQ{
		e^{t\Delta}f = &\int_{|y|<\la^{k-1}} \frac c {t^{\frac 3 2}} e^{-\frac{|x-y|^2}{4t}} f(y)\,dy+ \int_{y\in A_k^*}\frac c {t^{\frac 3 2}} e^{-\frac{|x-y|^2}{4t}} f(y)\,dy 
		\\&+ \int_{|y|\geq \la^{k+2}}\frac c {t^{\frac 3 2}} e^{-\frac{|x-y|^2}{4t}} f(y)\,dy
		\\=:&I_1+I_2+I_3.
	}
Beginning with a bound for $I_1$,
	\EQ{
		\| I_1 \|_{L^\I(A_k)}&\leq \frac c {t^{\frac 3 2}} e^{-\frac{\la^{2k}}{4t}} \| f\|_{L^1(|y|< \la^{k-1})} 
		\\&\leq \frac c {t^{\frac 3 2}} e^{-\frac{\la^{2k}}{4t}} \sum_{k'\in \Z: k'\leq k-2} \| f \|_{L^1(A_{k'})}
		\\&\lesssim_\la \frac 1 {t^{\frac 3 2}} e^{-\frac{\la^{2k}}{4t}} \sum_{k'\in \Z: k'\leq k-2} \la^{k'(3-\sigma)} \| f \|_{L^q(A_{1})}
		\\&\lesssim_\la \frac 1 {t^{\frac 3 2}} e^{-\frac{\la^{2k}}{4t}} \la^{(k-2)(3-\sigma)} (k-2)\| f \|_{L^q(A_{1})}.
	}
As $k\to\I$, the Gaussian dominates any algebraic growth. Hence
	\[
	\| I_1 \|_{L^\I(A_k)}\lesssim_\la \la^{k(\frac 3 q -\sigma)}\| f \|_{L^q(A_1)}.
	\]
For $I_2$, by Young's inequality and DSS scaling,
	\EQ{
		\| I_2 \|_{L^\I(A_k)} &\leq \| t^{-\frac 3 2}e^{-\frac{|\cdot|^{2}}{4t}}\|_{L^{q_*}} \| f \chi_{A_k}\|_{L^q}
		\lesssim_\la \la^{k(\frac 3 q -\sigma)} \| f \|_{L^q(A_{1})}.
	}
Note that this term determines the power of $R$ in the lemma's statement.
	
Finally, for $I_3$, we sum over the annuli $A_{k'}$ and find
	\EQ{
		\| I_3 \|_{L^\I(A_k)} &\leq \sum_{k'\geq k+2} \frac c {t^\frac 3 2} e^{-\frac{\la^{2k'}}{4t}}|A_{k'}|^{1-\frac 1 q} \|f\|_{L^q(A_{k'})} 
		\\&\lesssim_\la \sum_{k'\geq k+2} \frac 1 {t^\frac 3 2} e^{-\frac{\la^{2k'}}{4t}}|A_{k'}|^{1-\frac 1 q} \la^{k'(\frac 3 q -\sigma)} \|f\|_{L^q(A_{1})}.
	}
Again, the Gaussian dominates any algebraic growth, and we conclude 
	\[
	\| I_3 \|_{L^\I(A_k)}\lesssim_\la \la^{k(\frac 3 q -\sigma)}\| f \|_{L^q(A_1)}.
	\]
	This completes the proof.
\end{proof}
The next lemma states that the decay condition just established for the solution to the heat equation extends to higher Picard iterates.
\begin{lemma}\label{lemma:Picard.decay}Fix $3<q\leq \I$. 
Assume, for all $(x,t)\in \R^3\times (0,\I)$, that
	\[
	P_0(x,t)\lesssim \frac 1 { \sqrt t^{\frac 3 q}(|x|+\sqrt t)^{1-\frac 3 q}}. 
	\]
Then, for all $k\in \N$ and all $(x,t)\in \R^3\times (0,\I)$,
	\[
	P_k(x,t)\lesssim_{k,P_0} \frac {1} {\sqrt t^{\frac 3 q}(|x|+\sqrt t)^{1-\frac 3 q}}. 
	\]
\end{lemma}
\begin{proof}
This follows inductively from the fact that 
	\[
	P_k -P_{k-1}= B(P_{k-1},P_{k-1}),
	\]
which, by \eqref{ineq:Tsai.integral}, decays algebraically faster than $P_{k-1}$. This implies $P_k$ and $P_{k-1}$ have the same decay.
\end{proof}

We may now prove the following corollary to \cref{lemma:heat.equation.pointwise.decay}. These decay rates are central to our proof of $C^{1,\al}$ decay.

\begin{corollary}\label{cor:DecayFractionLaplacian}Let $0<\al< 1$ and $m\in N_0^3$ be a multi-index with $|m|\in \{0,1\}$.	If  ${u_0} \in C^{|m|,\al}_\loc(\R^3\setminus\{0\})$ is DSS, then, for all $\be \in (0,\al)$,
	\EQ{\label{eq:lau0LI}D^m\La^{\be} u_0 \in L^\I_\loc(\R^3\setminus \{0\}),}
	and, furthermore,
	\[
	\sup_{t\in [1,\la^2]}	|D^m\La^\be e^{t\Delta} u_0 (x,t)| \lesssim_{\la,u_0,\al,\be} \frac 1 {(|x|+1)^{1+m+\be}}.
	\]
\end{corollary}
\begin{proof}By scaling, it will suffice to show 
	\EQ{\label{eq:lau0LI}D^m\La^{\be} u_0 \in L^\I(A_1).}
Let $x\in A_1$. We have
	\EQ{  \La^{\be} D^m u_0(x)\lesssim&\bigg( \int_{|x-y|\le\la^{-1}} +\int_{|x-y|>\la^{-1}} \bigg)\frac {1}{|x-y|^{3+\beta}}(D^mu_0 (x)-D^mu_0 (y)) \, dy\\
		=:& J_1(x,t)+J_2(x,t)
		.}
	For the near-field in $J_1$, 
	\EQ{|J_1|&= \int_{|x-y|\le\la^{-1}}\frac 1 {|x-y|^{3+\beta}}(D^mu_0 (x)-D^mu_0 (y))\,dy\\
		&\lesssim\|u_0\|_{C^{|m|,\al}(A_1^*)}\int_{|x-y|\le\la^{-1}} \frac 1 {|x-y|^{3+\beta-\al}} \, dy\\
	&\lesssim_{\la,u_0,\al,\be} 1,
	}
	
	Next, for $J_2$, by the decay for $u_0$, $|D^m u_0|(x)\lesssim_{u_0,\la} |x|^{-|m|-1}$,
	\EQ{|J_2|&\leq \int_{|x-y|>\la^{-1}} \frac {|D^mu_0 (y)| +|D^mu_0 (x)|}{|x-y|^{3+\be}} \, dy\\
		&\lesssim_{u_0,\la} \int_{|x-y|>\la^{-1}} \frac {1}{|x-y|^{3+\be}} \frac {1}{|y|^{|m|+1}} \, dy
+ \int_{|x-y|>\la^{-1}} \frac {1}{|x-y|^{3+\be}} \, dy
\\
		&\lesssim_{u_0,\la} \left(\int_{|x-y|>\la^{-1}, |y|>\la^{-1}} +\int_{|y|\le\la^{-1}}\right)\frac {1}{|x-y|^{3+\be}} \frac {1}{|y|^{|m|+1}} \, dy+1\\
		&\lesssim_{u_0,\la} \int_{|x-y|>\la^{-1}, |y|>\la^{-1}} \frac {1}{|x-y|^{3+\be}} \frac {1}{|x|^{|m|+1}} \, dy+\int_{|y|\le\la^{-1}} \frac {1}{|x|^{3+\be}}\frac {1}{|y|^{|m|+1}} \, dy+1\\
		&\lesssim_{u_0,\la}1  
		,} 
	so long as $\be>0$ and $|m|\in\{0,1\}$. Therefore $D^{m} \La^{\be} u_0(x) \in L^\I_{\loc}(A_1)$. By  DSS scaling we   extend this to $\R^3\setminus\{0\}$ with the calculation 
	\EQ{
		\La^\be D^mu_0 (   x ) &= \la^{3k+\be k} \int \frac { \la^{(1+|m|)k}D^mu_0(\la^k x  ) - \la^{(1+|m|)k}D^mu_0(\la^k y)} {|\la^kx-\la^ky |^{3+\be}} \,dy
		\\& = \la^{(\be +1+|m|)k}    \int \frac { D^m u_0(\la^k x  ) -  D^mu_0(y)} {|\la^k x-y |^{3+\al}} \,dy  
		\\&=  \la^{(\be +1+|m|)k}  (\La^\be D^m u_0)( \la^k x).
	}
	Since $\nb^m\La^\be e^{t\Delta} u_0  =  e^{t\Delta} \nb^m \La^\be u_0 $, we use Lemma \ref{lemma:heat.equation.pointwise.decay} with $\sigma =1+m+\be$ and $q=\I$ to conclude.
\end{proof}

\begin{remark}\label{remark:12/22/21}
	Based on the same idea as in the above corollary we can prove that if ${u_0}$ is DSS and $\La^{\be}{u_0}\in L^\I_{\loc}(\R^3\setminus \{0\})$, for some $0<\be<2$, then, 
	\EQ{\label{ineq:LaaP0}
		\sup_{t\in [1,\la^2]}	|\La^\be e^{t\Delta} {u_0} (x,t)| \lesssim \frac 1 {(|x|+1)^{1+\be}}.
	}
	Note that Lemma \ref{lemma:decayCaloricFractional} leaves out the case $\al =\be$. This is anticipated because the analogous result for $\al = \be$ would imply that for $f\in C^\al_\loc(\R^3\setminus \{0\})\cap DSS$, 
	\[
	\La^\al e^{t\Delta}f (x) \lesssim \frac 1 {(\sqrt t+|x|)^{1+\al}},
	\]
	which suggests $\La^\al f \lesssim |x|^{-1-\al}$ and, in particular, that $\La^\al f \in L^\I_\loc (\R^3\setminus \{0\})\cap DSS$---this is exactly the initial data that leads to \eqref{ineq:LaaP0}. This should not be true for general elements of $C^\al_\loc(\R^3\setminus \{0\})\cap DSS$ due to the equivalence $C^\al = B^{\al}_{\I,\I}$ which, in our setting, corresponds to $\La^\al f\in (B^0_{\I,\I})_\loc (\R^3\setminus \{0\})\cap DSS$ which is \textit{strictly weaker} than $\La^\al f \in L^\I_\loc (\R^3\setminus \{0\})\cap DSS$. Therefore, we expect that $\al<\be$ is necessary in Lemma \ref{lemma:decayCaloricFractional}.
\end{remark}

The next lemma shows that derivatives enjoy an improved decay rate when the initial data is H\"older continuous. 

\begin{lemma}\label{lemma:heat.equation.pointwise.decay.alpha}Assume $f\in C^\al_\loc(\R^3\setminus \{0\})$, where $0<\al< 1$, and satisfies, for some $\la >1$, that
		\[
		\la^{\sigma k} f(\la^k x) = f(x),
		\] 
where $\sigma<3$ (note that $\sigma=1$ corresponds to being DSS). Then,
	\[
	\sup_{t\in [1,\la^2]} \| \nb e^{t\Delta} f \|_{L^\I(B_R^c)} \lesssim_\la  \|f\|_{C^\al(A_0)} R^{-(\sigma+\al)} .
	\]
\end{lemma}
\begin{proof}
It suffices to prove that, for $t\in [1,\la^2]$, 
	\[
	\sup_{t\in[1,\la^2]}\| \nb e^{t\Delta}f \|_{L^\I (A_k)}\lesssim \la^{-(\sigma+\al)k }  \|f\|_{C^\al(A_0)},
	\]
where the suppressed constants are independent of $k$. Let $x\in A_k$. Note that 
	\EQ{\pd_i e^{t\Delta}f(x) = &\int_{\R^3} \frac c {t^{\frac 5 2}} (x_i-y_i) e^{-\frac{|x-y|^2}{4t}} f(y)\,dy.}
Since $(x_i-y_i) e^{-\frac{|x-y|^2}{4t}}$ is mean zero on spheres centered at $x$, 
	\EQ{
		\pd_i e^{t\Delta}f(x) = &\int_{|x-y|<\la^{k-1}} \frac c {t^{\frac 5 2}} (x_i-y_i) e^{-\frac{|x-y|^2}{4t}} (f(y)-f(x))\,dy
		\\&+ \int_{x-y\in A_k^*}\frac c {t^{\frac 5 2}} (x_i-y_i) e^{-\frac{|x-y|^2}{4t}} f(y)\,dy
		\\&+ \int_{|x-y|\geq \la^{k+2}}\frac c {t^{\frac 5 2}} (x_i-y_i) e^{-\frac{|x-y|^2}{4t}} f(y)\,dy
		\\=:&\,I_1(x,t)+I_2(x,t)+I_3(x,t).
	}
For $I_1$, because $x,y\in B_{\la^{k-1}}$, it follows that
	\EQN{|f(y)-f(x)|=\frac{\la^{-\sigma k}|f(\la^{-k} x)-f(\la^{-k} y)|}{\la^{k\al}|\la^{-k} x-\la^{-k} y|^\al}|x-y|^\al\lesssim [f]_{C^\al(A_0)} \frac{1}{\la^{k(\sigma+\al)}}|x-y|^\al,}  
and therefore
	\EQN{|I_1|(x,t) \lesssim \frac{c}{t^{\frac 5 2}}\frac {[f]_{C^\al(A_0)}}{\la^{k(\sigma+\al)}}\int_{|x-y|<\la^{k-1}} |x-y|^{1+\al}e^{-\frac{|x-y|^2}{4t}} \,dy.}
Because the integral above is bounded independently of $k$ and we are taking $t\in [1,\la^2]$, we have
	\EQN{\sup_{t\in[1,\la^2]} \|I_1\|_{L^\I}(t) &\lesssim_{\al,\la} \frac {[f]_{C^\al(A_0)} }{\la^{k(\sigma+\al)}},}
which determines the power of $R$ in the lemma's statement.
		
For $I_2$, because $x\in A_k$ and $x-y\in A_k^*$ we know $|y|<\la^{k+2}+\la^{k+1}$, and so
	\EQ{\| I_2 \|_{L_x^\I(A_k)}(t)&\leq \frac c {t^{\frac 5 2}} \la^{k+2}e^{-\frac{\la^{2k-2}}{4t}} \| f\|_{ L^1\left(|y|<\la^{k+1}(\la+1)\right)}.
	}
Then, because $|f(y)|\lesssim_\la \frac{\|f\|_{L^\I(A_0)}}{|y|^\sigma}$, 
	\EQ{\label{eqn:I2L1bound}\|f\|_{L^1\left(|y|<\la^{k+1}(\la+1)\right)}\lesssim_\la   \|f\|_{L^\I(A_0)} \int_{|y|<\la^{k+1}(\la+1)} \frac{1}{|y|^\sigma} \,dy \lesssim_\la \la^{(3-\sigma)k} \|f\|_{L^\I(A_0),}
	} 
given $\sigma<3$. Therefore,
	\EQ{\sup_{t\in[1,\la^2]}\| I_2 \|_{L_x^\I(A_k)}(t)&\lesssim_\la \la^{(4-\sigma)k} \|f\|_{L^\I(A_0)} \frac c {t^{\frac 5 2}}e^{-\frac{\la^{2k-2}}{4t}} .}
As $k\to\I$, the Gaussian dominates any algebraic growth. Hence,
	\[
	\| I_2 \|_{L^\I(A_k)}\lesssim_\la \la^{-(\sigma+\al)k} \|f\|_{L^\I(A_0)}.
	\]
Finally, for $I_3$, we sum over the annuli $A_{k'}$ and find 
	\EQ{
		\| I_3 \|_{L^\I(A_k)}(t) &\leq \sum_{k'\geq k+2} \frac c {t^\frac 5 2} \la^{k'} e^{-\frac{\la^{2k'}}{4t}} \|f\|_{L^1(A_{k'}^*)} 
		\\&\leq \sum_{k'\geq k+2} \frac c {t^\frac 5 2} e^{-\frac{\la^{2k'}}{4t}} \la^{(4-\sigma)k' }\|f\|_{L^\I(A_0)},
		} 
where we used DSS scaling. Again, the Gaussian dominates any algebraic growth so the preceding series is summable. We conclude 
	\[
	\sup_{t\in[1,\la^2]}\| I_3 \|_{L^\I(A_k)}\lesssim_\la \la^{-(\sigma+\al)k }\|f\|_{L^\I(A_0)}.
	\]
\end{proof}

\cref{lemma:heat.equation.pointwise.decay.alpha} suggests a similar result should hold for fractional derivatives of order between $\al$ and $1$. In order to prove such a result, we will need the following elementary lemma. 

\begin{lemma}[Decay of fractional derivatives of Schwarz functions] Let $\Ga$ be Schwarz and $0<\be<1$. Then 
	\[
	| \La^\be \Ga |(x) \lesssim_\Ga \frac 1 {(1+|x|)^{3+\be}}.
	\]
\end{lemma}
\begin{proof} Fix $\la=2$ in the definition of $A_k$.
Fix $k$ so that $x\in A_k $, where $k\in \N$. Then,
	\EQ{
		| \La^\be \Ga |(x) \lesssim&\bigg( \int_{|x-y|<2^{k-2}} + \int_{|x-y|\geq 2^{k-2};|y|>2^k} +\int_{|x-y|\geq 2^{k-2};|y|\leq 2^k} \bigg) \frac {\Ga(x)-\Ga(y)} {|x-y|^{3+\be}}\,dy
		\\\lesssim& { 2^{(1-\be)k} \|\nb \Ga\|_{L^\I(A_k)} + 2^{-k\be}\|\Ga\|_{L^\I(\{|x|>2^{k} \})} } + \int_{|x-y|\geq 2^{k-2};|y|\leq 2^k} \frac {\Ga(x) } {|x-y|^{3+\be}}\,dy
		\\&+\int_{|x-y|\geq 2^{k-2};|y|\leq 2^k} \frac {\Ga(y) } {|x-y|^{3+\be}}\,dy
		\\\lesssim& \underbrace{2^{(1-\be)k} \|\nb \Ga\|_{L^\I(A_k)} + 2^{-k\be+1}\|\Ga\|_{L^\I(\{|x|>2^{k} \})} }_{\text{rapid decay since $\Ga$ is Schwarz}} + \frac 1{2^{k(3+\be)}} \|\Ga\|_{L^1}
		\lesssim \frac 1 {|x|^{3+\be}}.
	}
The fact that $\La^\be \Ga \in L^\I ( \R^3)$ completes the proof.
\end{proof}
We now prove an analogue of Lemma \ref{lemma:heat.equation.pointwise.decay.alpha} for fractional derivatives.
\begin{lemma}\label{lemma:decayCaloricFractional}
Assume $f\in C^\alpha_\loc(\R^3\setminus \{0\})$, where $0<\alpha< 1$, and is DSS. Fix $\be\in (\al,1)$.
Then,
	\[
	\sup_{t\in [1,\la^2]} \| \La^\be e^{t\Delta} f \|_{L^\I(B_R^c)} \lesssim_\la  \|f\|_{C^\al(A_0)} R^{-(1+\alpha)} .
	\]
\end{lemma}

\begin{proof}
Let $\Ga_t$ be the Gaussian kernel of the heat equation at time $t$. We consider only $t\in [1,\la^2] $. We note that $\int \La^\be \Ga_t \,dx = 0$, for all $0<\be<2$ and $t>0$, since the Fourier transform of $\La^\be \Ga_t$ is zero at the origin. Fix $x\in A_k$, for some $k\in \N$. 
Let $K_t = \La^\be G_t$. 
Then,
	\EQ{
		\La^\be e^{t\Delta}f(x) &= \int K_t(x-y) (f(y)-f(x))\,dy
		\\&=\bigg(\int_{|x-y|<\la^{k-2} } +\int_{|x-y|\geq \la^{k-2} } \bigg) K_t(x-y) (f(y)-f(x))\,dy
		\\&=I_1+I_2.
	}
We have 
	\EQ{
		I_1 &\leq \| f\|_{C^\al( A_k)}  \int_{\R^3} |x-y|^\al K_t(x-y) \,dy
		\\&\lesssim_\la \la^{-k(1+\al)}\| f\|_{C^\al( A_0)}.
	}
We split $I_2$ into two integrals and treat each separately, starting with 
	\EQ{\label{ineq:2/22/21/a}
		\int_{|x-y|\geq \la^{k-2} } K_t(x-y)f(x)\,dy &\lesssim_\la \frac 1 {|x|} \int_{|x-y|\geq \la^{k-2} }  K_t(x-y) \,dy
		\\&
		\lesssim_\la \frac 1 {|x|^{1+\al}},
	}
since $|x|>1$ using the fact that $K_t$ is Schwarz. For the last term,
	\EQN{
		\int_{|x-y|\geq \la^{k-2} } K_t(x-y)f(y)\,dy &= \bigg( \int_{|x-y|\geq \la^{k-2} ;|y|>\la^{k-2}} +\int_{|x-y|\geq \la^{k-2} ;|y|\leq \la^{k-2}} \bigg) K_t(x-y)f(y)\,dy.
	}
By the argument in \eqref{ineq:2/22/21/a},
	\[
	\int_{|x-y|\geq \la^{k-2} ;|y|>\la^{k-2}} K_t(x-y)f(y)\,dy \lesssim_\la \frac 1 {|x|^{1+\al}}.
	\]
On the other hand,
	\EQ{
		\int_{|x-y|\geq \la^{k-2} ;|y|\leq \la^{k-2}} K_t(x-y)f(y)\,dy
		&\lesssim_{\|f\|_{L^\I(A_0)}} \frac 1 {\la^{k(3+\be)}} \int_{|y|\leq \la^{k-2}} \frac 1 {|y|}\,dy
		\\&\lesssim \frac 1 {\la^{k(1+\be)}} \lesssim \frac 1 {|x|^{1+\al}}.
	}
This completes the proof.
\end{proof}

The last lemma of this section is dedicated to the decay of commutators where the fractional Laplacian, and subsequently gradients, are applied to the two-tensor $P_0 \otimes P_0$. This lemma is used to eliminate logarithms in our decay results when $u_0 \in C^{\al}_{\loc}$ and $u_0 \in C^{1,\al}_{\loc}$.

	\begin{lemma}[Commutator decay]\label{lemma:CommutatorDecay}Fix $\al\in (0,1)$ and a multi-index $m \in \N_0^3$, $|m|\le1$. If $u_0\in C^{|m|,\al}_\loc (\R^3\setminus \{0\})$ is DSS, then, for $t\in [1,\la^2]$ and $\be\in \big(0,(2-|m|)\al\big)$,
		\[
		D^m \La^{\be} \cdot ({P_0} \otimes {P_0}) (x,t) \lesssim_{u_0,\la,\al,\be} {P_0}_i D^m\La^{\be} {P_0}_j (x,t)+ {P_0}_j D^m \La^{\be} {P_0}_i (x,t) + \mathcal O\bigg(\frac {1} {|x|^{2+|m| +\be}}\bigg),
		\] 
		for $1\leq i,j\leq 3$.
	\end{lemma}
	
	\begin{proof} A straightforward computation shows, for $|m|=0$,
		\EQN{
		\La^\be ({P_0}_i {P_0}_j) (x,t)&= {P_0}_i \La^\be {P_0}_j (x,t)+ {P_0}_j \La^\be {P_0}_i (x,t)
\\&\qquad +   \int \frac {({P_0}_i (x)-{P_0}_i(y))({P_0}_j(x)-{P_0}_j(y))} {|x-y|^{3+\be}} \,dy,
}
		and, for $|m|=1$, 
		\EQN{
			\pd_k\La^\be ({P_0}_i {P_0}_j) (x,t)&= {P_0}_i \pd_k\La^\be {P_0}_j (x,t)+ {P_0}_j \pd_k\La^\be {P_0}_i (x,t)\\&\qquad +   \int \frac {(x_k-y_k)} {|x-y|^{5+\be}}({P_0}_i (x)-{P_0}_i(y))({P_0}_j(x)-{P_0}_j(y)) \,dy.
		}
	In both cases,	to determine the decay of the principal-value integral, we need to bound an integral of the form
		\[ \int \frac {1} {|x-y|^{3+|m|+\be}}({P_0}_i (x)-{P_0}_i(y))({P_0}_j(x)-{P_0}_j(y)) \,dy\]
		Suppose $2\leq |x|$. We break the above integral into three parts: $|x-y|<\frac {|x|} 2$, $|x-y|>\frac {|x|} 2$ and $|y|>\frac {|x|} 2$, and $|y|<\frac {|x|} 2$. First, we consider $|x-y|<\frac {|x|} 2$, i.e.,
		\EQ{
			\int_{|x-y|<\frac {|x|} 2} \frac 1 {|x-y|^{3+|m|+\be}}({P_0}_i (x)-{P_0}_i(y))({P_0}_j(x)-{P_0}_j(y)) \,dy.
		}
		We will use the claim
		\EQ{\label{eq:P0Ca}
			\| {P_0} \|_{C^{\max\{|m|,\al\}} (A_1)}(t) \lesssim_{u_0} 1,
		}
		for $t>0$.  To verify this claim, let $\phi$ be a smooth function with the following properties: $\supp \phi \subset B_{\la^{-2}}$, $\phi|_{B_{\la^{-3}}} \equiv 1$, and $\supp \nb\phi \subset \{x: \la^{-3}\le |x|\le \la^{-2}\}$. Write $P_0 = e^{t\De} (u_0 (1-\phi)) + e^{t\De} (u_0 \phi).$ Because $u_0 \in C^{|m|,\al}_{\loc}(\R^3\setminus \{0\})$ it follows $u_0 (1-\phi) \in C^{|m|,\al}(\R^3)$ and so $e^{t\De} (u_0 (1-\phi)) \in C^{|m|,\al}(\R^3)$ for $t>0$.
		
		Next for the local part, if $x\in A_1$, we can bound $e^{t\De} (u_0 \phi)$ in $C^{1}(A_1)$ for $t>0$ as follows
		\EQ{\pd_i e^{t\De}(u_0\phi)(x)&\lesssim \int_{\R^3} \frac{(x_i-y_i)}{t^{\frac 5 2}} e^{-\frac{|x-y|^2}{4t}} (u_0 \phi) (y) \, dy\\
			&\lesssim \frac{\la+\la^{-2}}{t^{\frac 5 2}} e^{-\frac{(\la+\la^{-2})^2}{4t}} \int_{\R^3}(u_0 \phi) (y) \, dy\\
			&\lesssim_{u_0} \frac{\la+\la^{-2}}{t^{\frac 5 2}} e^{-\frac{(\la+\la^{-2})^2}{4t}} \int_{B_{\la^{-2}}}\frac 1 {|y|} \, dy\\
			&\lesssim_{u_0} \frac{\la+\la^{-2}}{t^{\frac 5 2}} e^{-\frac{(\la+\la^{-2})^2}{4t}} \la^{-4}\\
			&\lesssim_{u_0,\la} 1.}
		Therefore,
		\[\|\nb e^{t\De}(u_0\phi)\|_{L^\I(A_1)}(t) <\I,\]
		for $t>0$.Thus, $e^{t\De}(u_0\phi)\in C^1(A_1) \subset C^\al (A_1)$ for $t>0$, which implies the conclusion\eqref{eq:P0Ca}.
		
		By DSS scaling, if $y$ is in $B_{\frac {|x|} 2}(x)$, then, uniformly in $t$, we have
		\[
		\frac {|{P_0}(x,t)-{P_0}(y,t)|}{|x-y|^{\al}} \lesssim_\la \frac 1 {|x|^{1+\al}} [P_0]_{C^{\al} (A_1)}.
		\]
		and, for $|m|=1$,
		\[
		\frac {|{P_0}(x,t)-{P_0}(y,t)|}{|x-y|} \lesssim_\la \frac 1 {|x|^{2}} [P_0]_{C^{1} (A_1)}.
		\]
		Hence,
		\EQ{
			& \int_{|x-y|<\frac {|x|} 2} \frac 1 {|x-y|^{3+|m|+\be}}({P_0}_i (x)-{P_0}_i(y) )({P_0}_j(x)-{P_0}_j(y)) \,dy\\ 
			&\lesssim_{\la,u_0}\int_{|x-y|<\frac {|x|} 2} \frac 1 {|x-y|^{3+|m|+\be-2\max\{|m|,\al\}}}\frac{{P_0}_i (x)-{P_0}_i(y) }{|x-y|^{\max\{|m|,\al\}}}\frac{{P_0}_j(x)-{P_0}_j(y)}{|x-y|^{\max\{|m|,\al\}}}\,dy\\
			&\lesssim_{\la,u_0} \frac 1 {|x|^{2+2\max\{|m|,\al\} }} \int_{|x-y|<\frac {|x|} 2} \frac 1 {|x-y|^{3+|m|+\be-2\max\{|m|,\al\}}}\,dy
			\\&\lesssim_{\la,u_0} \frac 1 {|x|^{2+|m|+\beta}},
		}
		provided $|m|+\be<2\max\{|m|,\al\}$. This is clearly satisfied for $\be <(2-|m|)\al$.
		
		Next, we consider the case when $|x-y|>\frac {|x|} 2$ and $|y|>\frac {|x|} 2$. Because $|{P_0}(y,t)|\lesssim_{\la,u_0} \frac 1 {|y|}$ and $|y|>|x|/2$, we find
		\EQ{
			&\int_{|x-y|>\frac {|x|} 2;|y|>\frac {|x|} 2} \frac 1 {|x-y|^{3+|m|+\beta}}({P_0}_i (x)-{P_0}_i(y) )({P_0}_j(x)-{P_0}_j(y))\,dy 
			\\&\lesssim_{\la,u_0} \frac 1 {|x|^{2}} \int_{|x-y|>\frac {|x|} 2;|y|>\frac {|x|} 2}\frac 1 {|x-y|^{3+|m|+\beta}} \,dy
			\\& \lesssim_{\la,u_0} \frac 1 {|x|^{2+|m|+\be}},
		} 
		for $\be+|m|>0$, clearly satisfied for $\be>0.$
		The final region is $|y|<\frac {|x|} 2$. By the same decay for ${P_0}(x,t)$, we find
		\EQ{&\int_{|y|<\frac {|x|} 2} \frac 1 {|x-y|^{3+|m|+\beta}}({P_0}_i (x)-{P_0}_i(y) )({P_0}_j(x)-{P_0}_j(y))\,dy\\
			&\lesssim_{\la,u_0} \int_{|y|<\frac {|x|} 2} \frac { 1} {|x-y|^{3+|m|+\beta}}\frac 1 {|y|^{2}}\,dy\\
			&\lesssim_{\la,u_0} \frac { 1} {|x|^{3+|m|+\be}}\int_{|y|<\frac {|x|} 2}\frac{1}{|y|^{2}}\,dy\\
			&\lesssim_{\la,u_0}\frac { 1} {|x|^{2+|m|+\be}},}
		and our proof is complete.
	\end{proof}

\section{Decay of DSS Navier-Stokes flows}
\subsection{Decay rates when $u_0\in L^{3<q\leq \I}_\loc(\R^3\setminus \{0\})$ }

Here we prove Theorem \ref{theorem:Lpdecay1} but first prove a supporting lemma. Our main tool is the local smoothing result Theorem \ref{theorem:JSlocalsmoothing} of Jia and \v Sver\' ak which we use alongside the following lemma on the decay of $u$ above a space-time parabola. Recall that $A_k^* = \{x: \la^{k-1}\leq |x|<\la^{k+2}\}$.

\begin{lemma}\label{lemma:DSSLqDecay}
Let $\ga>0$ be given.
Suppose $u$ is a $\la$-DSS, local energy solution to \eqref{eq:NS} with divergence free, DSS data in $E^2$. Assume $u$  satisfies 
	\[
	\max\{\sup_{0<s<T} \| u \|_{L^q(A_{0}^*)},\sup_{0<s<T} \| p \|_{L^{q/2}(A_{0}^*)} \}<\ga,
	\]
for some $T>0$.
Then, for any $l\in N_0$,
	\[
	|u(x,t)| \lesssim_{\la,l,u_0} \frac {\ga } {\sqrt t^{|l|+\frac 3 q} (|x|+\sqrt t)^{1-\frac 3 q}}\text{ for } |x| \geq R_0 \sqrt t.
	\]
\end{lemma}
\begin{proof}Suppose $x_0\in A_k$. By scaling,
	\EQN{
		\int_0^1 \int_{B_1(x_0)} | u|^3\,dx\,dt &\lesssim_\la |x_0|^{3(\frac 3 q-1)} \sup_{0<s\lesssim_{\la}|x_0|^{-2}} \bigg( \int_{B_{\frac 1 {\la^k}}(\frac{x_0}{\la^k})}|u|^q \,dx \bigg)^{\frac 3 q} \\&\lesssim_\la |x_0|^{3(\frac 3 q-1)} \ga^3,
	}
provided $|x_0|^{-2}\lesssim_\la T$.
Similarly for the pressure
	\[
	\int_0^1 \int_{B_1(x_0)} | p|^{\frac 3 2}\,dx\,dt \lesssim_\la |x_0|^{3(\frac 3 q-1)} \ga^3.
	\]
So, provided $C |x_0|^{3(\frac 3 q-1)} \ga^3 <\e_{*}$ and $|x_0|^{-2}\lesssim_\la T$, we obtain from Lemma \ref{lemma:ereg} that, for every $l\in \N_0$,
	\[\sup_{3/4<s<1} \| \nb^l u\|_{L^\I(B_{\frac 1 2}(x_0))} \lesssim_{\la,l} |x_0|^{(\frac 3 q-1)} \ga.\]
Implying, since $|x|\gg1$, for $t\in [\frac 3 4, 1]$ and any $l\in \N_0$,
	\[
	|\nb^l u(x, t)| \lesssim_{\la,l} \frac {\ga} {(|x| +1)^{1-\frac 3 q}}. 
	\]
This procedure can be repeated finitely many times (depending on $\la$) on cylinders of the form $B_1(x_0)\times [\delta,\delta+1]$ until we obtain the estimate for all $t\in [ \frac 3 4,\frac 3 4+\la ^2]$, which contains a full `period' of discrete self-similarity of $u$. Then, because $u$ is regular for $|x|>R_0\sqrt t$, by increasing $C$ and using DSS scaling this decay extends to the entire sub-parabaloid region of regularity in re-scaled form as 
	\[
	|u(x,t)|\lesssim_{\la,u_0} \frac {\ga} { \sqrt t^{|l|+\frac 3 q} (|x|+\sqrt t)^{1-\frac 3 q}}.
	\]	
\end{proof}

\begin{proof}[Proof of Theorem \ref{theorem:Lpdecay1}]
\textit{Part 1:} 
Assume $u$ is a DSS, local energy solution and $u_0\in L^q_\loc(\R^3\setminus \{0\})$, for some $q>3$. By repeatedly applying \cref{theorem:JSlocalsmoothing}, 
	\EQ{u-a\in C^\gamma_{par}(\overline {A_0} \times [0,T]),}
where $T$ is as in \cref{theorem:JSlocalsmoothing}. So, $u-a\in L^\I( 0,T; L^q (A_{0}))$. Since $a$ is also in $L^\I([0,T];L^q (A_{0}))$ by sub-critical local well-posedness in $L^q$, $u\in L^\I( 0,T; L^q (A_{0}))$.
	
We now obtain a bound for the pressure in $ L^\I( 0,T; L^{q/2} (A_{0}))$. For $x\in A_0$, break the pressure into the following pieces:
	\EQ{
		p(x,t)= [-\De^{-1}\div\div]_{ij} (u_iu_j ) (x,t)=(p_1+p_2+p_3)(x,t),
	}
where
\EQ{
&p_1=[-\De^{-1}\div\div]_{ij} (u_iu_j \chi_{B_{\la^{-1}}}  ) 
\\& p_2=[-\De^{-1}\div\div]_{ij} (u_iu_j \chi_{A_0^*}  )
\\&p_3= [-\De^{-1}\div\div]_{ij} (u_iu_j \chi_{B^c_{\la^{2}}}  ) ,
}
and the balls are taken to be centered at the origin. By \textit{a priori} estimates for local energy solutions,
\EQ{
\|[-\De^{-1}\div\div]_{ij} (u_iu_j \chi_{B_{\la^{-1}}}  ) \|_{L^\I (A_0 \times (0,T])}\lesssim_\la \|u_0\|_{L^2_\uloc}^2.
}
Then, by Calderon-Zygmund theory $p_2\in L^\I (0,T;L^{q/2} (A_0) )$.  The last term is more involved:
\EQ{
|p_3(x,t)|&\lesssim  \sum_{k\geq 2}\frac 1 {2^{3k}} \int_{A_k} |u(y,t)|^2\,dy 
\\&\lesssim \sum_{k\geq 2}\frac 1 {2^{3k}}  2^{3k(1-\frac 2 q)}\bigg(\int_{A_k} |u(y,t)|^q\,dy\bigg)^{\frac 2 q} 
\\&\lesssim  \sum_{k\geq 2}\frac 1 {2^{3k \frac 2 q}}  \bigg(2^{-qk+3k}  \int_{A_0} |u(y,2^{-k}t)|^q\,dy\bigg)^{\frac 2 q} 
\\&\lesssim\bigg( \sum_{k\geq 2} 2^{-2k}\bigg) \sup_{0<t<T}\|u(\cdot,t) \|^2_{L^{q}(A_0)},
}
and so $p_3\in L^\I (A_0\times (0,T] )$. Hence, $p\in L^\I (0,T;L^{q/2} (A_0) )$. By scaling and   modifying $T$ in a fashion depending on $\la$ to get these estimates over $A_0^*$, we have shown that $u$ and $p$ satisfy the assumption of  \cref{lemma:DSSLqDecay} and, consequently, obtain the desired decay rate for $u$ and its derivatives.

\bigskip\noindent 

\textit{Part 2:} We now show improved decay rates for $\td u=u-e^{t\De}u_0$. Assume $|x|>2\la^2 R_0$ and $t\in [1,\la^2]$, then
	\EQ{|\td u (x,t) |
	&\lesssim   \int_0^t \int_{B_{R_0 \sqrt s}}\frac{1}{(|x-y|+\sqrt{t-s})^4}|u_iu_j|\,dy\,ds \\ 
	&+ \int_0^t\int_{B_{R_0\sqrt s}^c}\frac{1}{(|x-y|+\sqrt{t-s})^4}\frac{1}{\sqrt{s}^{\frac{6}{q}}(|y|+\sqrt{s})^{2-\frac{6}{q}}}\,dy\,ds }
by the pointwise upper bound \eqref{ineq:pointwiseBound.q} for $u$ in the region of regularity. Using \cref{lemma.intbound1} on the far-field integral,
	\EQ{\int_0^t\int_{B_{R_0\sqrt s}^c}\frac{1}{(|x-y|+\sqrt{t-s})^4}\frac{1}{\sqrt{s}^{\frac{6}{q}}(|y|+\sqrt{s})^{2-\frac{6}{q}}}\,dy\,ds \lesssim_{\la}\frac{1}{(|x|+1)^{2-\frac{6}{q}}}.}
Next, using the fact that $|x|\lesssim |x-y|$ in the near-field integral
	\EQ{\int_0^t \int_{B_{R_0\sqrt s}}\frac{1}{(|x-y|+\sqrt{t-s})^4}|u_iu_j|\,dy\,ds
	\lesssim_\la \frac{1} {(|x|+1)^4} \int_0^t\int_{B_{2\la R_0}}|u|^2 \,dy.}
Now, because $u\in L^\I(0,\la^2; L^2_{\uloc}(\R^3))$ by \emph{a priori} local energy estimates, and since $R_0$ is fixed,
	\EQ{\frac{1} {(|x|+1)^4}\int_0^t\int_{B_{2R_0\la}}|u|^2 \,dy\,ds
	&\lesssim_{R_0,\la} \frac{1} {(|x|+1)^4}\|u\|^2_{L^\I(0,\la^2; L^2_{\uloc}(\R^3))}.}
This gives us the bound 
	\EQ{|\td u(x,t)|&\lesssim_{\la,R_0,u_0} \frac1{(|x|+1)^4}  +\frac{1}{(|x|+1)^{2-\frac{6}{q}}}.}
The weakest decay rate is the last term.  Re-scaling, using the fact that $\td u$ is DSS and increasing the suppressed constant to fill in the rest of the region of regularity $|x|\geq R_0\sqrt t$ gives us the asserted estimate.  
  
\end{proof}

\begin{proof}[Proof of \cref{theorem:pwdecay}]

For induction, we define for $k\in \N_0$,
\[
a_k = (k+2)\bigg(1-\frac{3}{q}\bigg) = a_{k-1}+1-\frac 3 q.
\]
Note that $a_k<4$ precisely when $k<\frac {4q} {q-3}-2$.
Our base case is 
\[
|u-P_0|(x,t)\lesssim_{\la,R_0,u_0} \frac {\sqrt t^{ 2 - \frac 6 q} } {\sqrt t (|x|+\sqrt t)^{2-\frac 6 q}},
\]
which is implied by Theorem \ref{theorem:Lpdecay1}.

To start, assume that 
	\EQ{|u-P_k|(x,t) \lesssim_{k,\la,R_0,u_0}\frac{\sqrt{t}^{a_k}}{\sqrt{t}(|x|+\sqrt{t})^{a_k}},}
for $|x|\geq R_0\sqrt t$.

Using the triangle inequality, bilinearity of $B(\cdot,\cdot)$, and $P_k = P_0 - B(P_{k-1},P_{k-1})$, we can bound $|u-P_{k+1}|$ by
	\EQ{|u-P_{k+1}| &= |P_0-B(u,u)-P_{k+1}|\\
		&=|B(u,u)-B(P_k,P_k)|\\
	&\lesssim |B(u,u-P_k)+B(u,P_k)-B(P_k,P_k)|\\
	&\lesssim |B(u-P_k,u-P_k)+B(P_k,u-P_k)+B(u-P_k,P_k)|.
}
Therefore,
	\EQ{|u-P_{k+1}|\lesssim&\int_0^t \int_{\R^3}\frac{1}{(|x-y|+\sqrt{t-s})^4}\big(|u-P_k|^2+|u-P_k||P_k|\big)\,dy\,ds\\
	\lesssim& \int_0^t \int_{B_{R_0\sqrt t}}\,\,\frac{1}{(|x-y|+\sqrt{t-s})^4}\big(|u-P_k|^2+|u-P_k||P_k|\big)\,dy\,ds\\
	&+\int_0^t \int_{B_{R_0\sqrt t}^c}\frac{1}{(|x-y|+\sqrt{t-s})^4}\big(|u-P_k|^2+|u-P_k||P_k|\big)\,dy\,ds\\
	=:&I_{k+1}(x,t)+J_{k+1}(x,t).}

Using \cref{lemma:Picard.decay} and \cref{lemma.intbound1},
	\EQ{J_{k+1}(x,t)\lesssim
&\int_0^t \int_{B_{R_0\sqrt t}^c}\frac{1}{(|x-y|+\sqrt{t-s})^4}\bigg(\frac{ s^{a_k-1}}{(|y|+\sqrt{s})^{2a_k}}+\frac{ \sqrt{s}^{a_k}}{\sqrt{s}^{1+\frac 3 q}(|y|+\sqrt{s})^{a_{k+1}}}\bigg)\,dy\,ds\\
	\lesssim& \frac{{\sqrt t}^{2a_k}}{\sqrt t(|x|+\sqrt{t})^{2a_{k}}}+\frac{\sqrt{t}^{a_{k+1}}}{\sqrt t(|x|+\sqrt{t})^{a_{k+1}}}.
}
It is easy to see that  the second term is decaying more slowly than the first. Hence,  we can increase the constant to   overcome the term with faster decay  to conclude
\[
	J_{k+1}(x,t) \lesssim_k \frac{\sqrt{t}^{a_{k+1}}}{\sqrt t(|x|+\sqrt{t})^{a_{k+1}}}.
\]

For $I_{k+1}$, 
%
	\EQ{I_{k+1}(x,t)&\lesssim \frac  t{(|x|+\sqrt t)^4} \big( \| u \|_{L^2(B_{\la R_0})}^2  + \| P_k\|_{L^2(B_{\la R_0})}^2   \big)
	\\&\lesssim_{\la,R_0,u_0,k}   \frac  1{(|x|+1)^4},
	}
for $t\in [1,\la^2]$.  Re-scaling, noting $I_{k+1}$ is DSS, leads to the better estimate 
\EQ{
	I_{k+1}(x,t)&\lesssim_{\la,R_0,u_0}  \frac{\sqrt{t}^3}{(|x|+\sqrt{t})^4}.
}

Combining these estimates and using DSS scaling we obtain
	\EQ{|u-P_{k+1}| &\lesssim_{k,\la,R_0, u_0}\frac{\sqrt{t}^{a_{k+1}}}{\sqrt{t}(|x|+\sqrt{t})^{a_{k+1}}}+\frac{\sqrt{t}^3}{(|x|+\sqrt{t})^4}.}
If $k+1<\frac{4q}{q-3}-2$,  then the term involving $a_{k+1}$ above dominates and our induction continues. Otherwise our induction terminates and the optimal   decay rate is reached, i.e., \EQ{|u-P_{k_q}| \lesssim_{k_0,\la,R_0,u_0 }\frac{\sqrt{t}^3}{(|x|+\sqrt{t})^4},} 
where  $k_q$ is chosen to be the  smallest natural number satisfying $a_{k_q}\geq 4$.

Lastly, we show the separation rate. Clearly, if two solutions $u$ and $v$ belong to $\mathcal{N}(u_0)$, then, for $|x|>R_0\sqrt t$, 
	\[|u-v|(x,t) =|u- P_{k_q}(u_0) - (v-P_{k_q}(u_0))|(x,t) \lesssim_{q,\la,R_0,u_0} \frac{\sqrt{t}^3}{(|x|+\sqrt{t})^4}.\]
\end{proof}

\subsection{Decay rates when $u_0\in C^{0<\al\leq 1}_\loc(\R^3\setminus \{0\})$ }

Now, we may show the improved decay rate for $\td u =u-e^{t\De}u_0$, where $u_0\in C^\alpha_{\loc}(\R^3\setminus \{0\})$.
\begin{proof}[Proof of \cref{theorem:tdu.pointwise.decay.CAl}]

We know, for any $(x,t) \in \R^4_+$,
	\EQ{\label{eq:setup}|\td u| =|B(u,u)| \le |B(u,\td u)|+|B(\td u,P_0)| + |B(P_0,P_0)|.}
Note that ${P_0}$, which decays like $u$, decays algebraically slower than $\td u = u-P_0$ and, therefore, the last term, heuristically, decays the slowest. To treat this term, we consider two cases: $\al<1$ and $\al=1$. 
	
The first case is $\al<1$. Let $\be\in (\al,\min(1,2\al))$. By standard properties of Fourier multipliers,
	\EQ{ B({P_0},{P_0})(x,t)&=\int_0^t\int_{\R^3} S(x-y,t-s)\nb \cdot ({P_0}\otimes {P_0})(y,s) \,dy \, ds\\
		&= \int_0^t \int_{\R^3} \nb \La^{-\be}S(x-y,t-s)\La^\be ({P_0}\otimes {P_0})(y,s)\,dy \, ds.}
By \cref{prop:DL.Oseen.est}, we have a bound for the fractional derivative of the Oseen Kernel
	\[\left|\nb \La^{-\be}S(x-y,t-s)\right| 
	\lesssim\frac{1}{(|x-y|+\sqrt{t-s})^{4-\be}}.\] 
Recall that \cref{lemma:CommutatorDecay} applied with $|m|=0$ implies that, for $0<\be<2\al$,
	\[\La^\be ({P_0}\otimes {P_0})(y,s)= {P_0}_i\La^\be {P_0}_j(y,s)+{P_0}_j\La^\be {P_0}_i(y,s)+ O\left(\frac 1 {(|y|+1)^{2+\be}}\right).\] 
So, by scaling, Lemma \ref{lemma:decayCaloricFractional}, and noting that $\al<\be$,
	\EQ{\La^\be ({P_0}\otimes {P_0})(y,s) \lesssim \frac{C }{(|y|+\sqrt{s})^{2+\al}}.}	
Therefore, 
	\EQ{
 |B({P_0},{P_0})|(x,t)&\lesssim \int_0^t \int_{\R^3} \frac{1}{(|x-y|+\sqrt{t-s})^{4-\be}}\frac 1 {(|y|+\sqrt{s})^{2+\al}}\,dy \, ds
\\&\lesssim \frac { t^{\frac {1+\al} 2}} {(|x|+\sqrt{t})^{2+\al}},
	} 
using a re-scaled version of \eqref{ineq:Tsai.integral}.
	
If $\al=1$, then we argue similarly using Lemma \ref{lemma:heat.equation.pointwise.decay.alpha} instead of Lemma \ref{lemma:decayCaloricFractional}. The main difference is that we cannot move any derivatives onto the kernel of the heat operator to shift away from a cubic power. This results in a logarithm when applying \eqref{ineq:Tsai.integral}. In particular, we obtain 
	\EQ{\label{ineq:bilinear.caloric}
		|B({P_0},{P_0})|(x,t) &\lesssim \frac { t } {(|x|+\sqrt{t})^{3}} \log(|x|+\sqrt t).
	} 

We now bound the remaining terms from \eqref{eq:setup}, namely $  |B(u,\td u)|+|B(\td u,P_0)| $. Assume $t\in [1,\la^2]$ and $|x|\geq 2\sqrt t R_0$. Then 
\EQ{
|B(u,\td u)|(x,t) &\lesssim \bigg(\int_0^t \int_{|y|\geq  \sqrt t R_0}  + \int_0^t \int_{|y|<\sqrt t R_0}  \bigg)  \frac 1 {(|x-y|+\sqrt {t-s})^4 } |u\td u|(y,s)\,dy\,ds
\\&=:I_1+I_2.
}
By Theorem \ref{theorem:Lpdecay1},
\EQ{
I_1\lesssim \int_0^t \int  \frac 1 {(|x-y|+\sqrt {t-s})^4 } \frac {\sqrt s} {(|y|+\sqrt s)^{3}} \,dy\,ds \lesssim \frac t {(|x|+\sqrt t)^3},
}
where we used  Lemma \ref{lemma.intbound1}. Above we worked with $t\in [1,\la^2]$ and extend this estimate to all $t$ by scaling.
On the other hand, by the \textit{a priori} bound \eqref{ineq.apriorilocal} we have that $\| u\|_{L^\I(0,\la^2; L^2(B_{\la R_0}))}\lesssim_{\la ,R_0} \|u_0\|_{L^2_\uloc}$ and by \cite{MaTe}, we have that
$\| P_0\|_{L^\I(0,\la^2; L^2(B_{\la R_0}))}\lesssim_{\la, R_0} \|u_0\|_{L^2_\uloc}$.
Hence, noting that $|x|\sim |x-y|$ due to our choice of $x$ and $y$, we have
\[
I_2\lesssim \frac 1 {|x|^4} \|u_0\|_{L^2_\uloc}^2.
\]
Since $|x|>1$ when $t\in [1,\la^2]$ by re-scaling we obtain 
\[
I_2(x,t)\lesssim \frac {t^{3/2}} {(|x|+\sqrt t)^4}.
\]
The estimate for $|B(\td u,P_0)|$ is basically identical since $P_0$ and $u$ have the same decay properties and bounds in $L^2_\uloc$.

\end{proof}

\begin{remark}
In the above proof for the case $\al=1$, the logarithm only appeared in the estimate  \eqref{ineq:bilinear.caloric} for $B(P_0,P_0)$. If this logarithm is necessary, and since $B(P_0,P_0)$ can be computed explicitly, it should be possible to construct an example demonstrating this necessity. On the other hand, if the log in \eqref{ineq:bilinear.caloric} can be removed, then the same is true in Theorem \ref{theorem:tdu.pointwise.decay.CAl}.
\end{remark}

\subsection{Decay rate when $u_0\in C^{1,0<\al\leq 1}_\loc(\R^3\setminus \{0\})$}

\begin{proof}[Proof of \cref{theorem:tdu.pointwise.decay.C1Al}] Let $t \in[1,\la^2]$ and $|x|>2 R_0\sqrt{t} $. First, decompose $\td u$ as 
	\EQ{\td u =& B(u,u)=B(u,\td u) + B(\td u,{P_0})  + B({P_0},{P_0})\\
		=&\int_0^t \int_{\R^3} \nb S (x-y,t-s) ({u }\otimes {\td u }   +{\td u }\otimes {P_0})(y,s)\,dy\,ds\\
		&+\int_0^t \int_{\R^3} \nb S (x-y,t-s)  ( P_0 \otimes P_0)(y,s)\,dy\,ds
		.}
In the proof of Theorem \ref{theorem:tdu.pointwise.decay.CAl} we showed that 
\EQ{
\int_0^t \int_{\R^3} \nb S (x-y,t-s) ({u }\otimes {\td u }   +{\td u }\otimes {P_0})(y,s)\,dy\,ds \lesssim \frac {t} {(|x|+\sqrt t )^3}.
}
The same estimate holds presently because we are making stronger assumptions on the data.
We now establish the same decay for $ B({P_0},{P_0})$. To accomplish this, we rewrite the integral for $B(P_0,P_0)$ as follows where $\be\in (0,1)$ will be specified momentarily:
\[
\int_0^t \int_{\R^3} \nb  S(x-y,t-s)  ( P_0 \otimes P_0)(y,s)\,dy= \int_0^t \int_{\R^3} \La^{-\be}\nb S(x-y,t-s)  \La^\be ( P_0 \otimes P_0)(y,s)\,dy\,ds.
\]
Note that
\[
\int_{\R^3} \nb\La^{-\be} S \,dx = 0,
\] 
since $\mathcal F (\nb\La^{-\be} S  ) (0)=0$ which uses the fact that $\be<1$. Using this, we re-write the above integral as
\EQ{\label{1.19.22.a}
	 &\int_0^t \left(\int_{|x-y|\leq \frac{|x|} 2} + \int_{|x-y|>\frac{|x|} 2}\right) \La^{-\be}\nb S(x-y,t-s)  \La^\be ( P_0 \otimes P_0)(y,s)\,dy\,ds\\
	&= \int_0^t \int_{|x-y|\leq \frac{|x|} 2} \La^{-\be}\nb S(x-y,t-s)   \bigg( \La^{\be} ( P_0 \otimes P_0)(y,s) - \La^{\be} ( P_0 \otimes P_0)(x,s) \bigg) \,dy\,ds\\
	&-\int_0^t \int_{|x-y|>\frac{|x|} 2} \La^{-\be}\nb S(x-y,t-s)  \La^\be ( P_0 \otimes P_0)(x,s) \,dy\,ds\\
	&+ \int_0^t \int_{|x-y|>\frac{|x|} 2} \La^{-\be}\nb S(x-y,t-s)  \La^\be ( P_0 \otimes P_0)(y,s)\,dy\,ds\\
	&=: J_1(x,t)+J_2(x,t)+J_3(x,t).}
We first bound $J_1$.   Using the mean value theorem and \cref{prop:DL.Oseen.est},  we have
\EQ{\label{ineq:C1aI}|J_1|(x,t) \lesssim & \int_0^t \int_{|x-y|\leq \frac{|x|} 2} \frac {|x-y|} {(|x-y|+\sqrt{t-s})^{4-\be}} \| \nb \La^{\be }   ( P_0 \otimes P_0)\|_{L^\I(B_{\frac{|x|} 2}(x))}(t)\,dy\,ds
	\\\lesssim & |x|^{\be} \sup_{0<s<t} \| \nb \La^\be  ( P_0 \otimes P_0)\|_{L^\I(B_{\frac{|x|} 2}(x))(s)}.
}
We further assume $\be \in (0,\al)$.  Since $u_0\in C^{1,\al}_\loc(\R^3\setminus\{0\})$,  by \cref{cor:DecayFractionLaplacian} we have
\EQ{\label{ineq:P0nbLaBDecay}
\sup_{s\in[1,\la^2]} |\nb \La^{\be} P_0(x,s)| \lesssim (|x|+1)^{-2-\be},
}
and
\EQ{\label{ineq:P0LaBDecay}
	\sup_{s\in[1,\la^2]} |\La^{\be} P_0(x,s)| \lesssim (|x|+1)^{-1-\be}.
}
Next, by \cref{lemma:CommutatorDecay}, \eqref{ineq:P0nbLaBDecay}, and the decay for $P_0$ 
\EQ{\label{ineq:P0nbCommDecay}
|\nb \La^{\be}  ( P_0 \otimes P_0)| (x,t) &\leq |{P_0}_i \nb\La^{\be} {P_0}_j | (x,t)+| {P_0}_j \nb \La^{\be} {P_0}_i |(x,t) + \mathcal O\bigg(\frac {1} {|x|^{3 +\be}}\bigg)\\
&\lesssim_{\la}(|x|+1)^{-3-\be}.
}
and likewise, using\eqref{ineq:P0LaBDecay},
\EQ{\label{ineq:P0CommDecay}
|	\La^{\be} ( P_0 \otimes P_0) |(x,t) &\leq | {P_0}_i \La^{\be} {P_0}_j |(x,t)+| {P_0}_j  \La^{\be} {P_0}_i |	 (x,t) + \mathcal O\bigg(\frac {1} {|x|^{2 +\be}}\bigg)\\
	&\lesssim_{\la,k}(|x|+1)^{-2-\be}.
}
Therefore, \eqref{ineq:C1aI} is bounded by
\[
|x|^{\be} (|x|+1)^{-3-\be}\lesssim (|x|+1)^{-3}.
\]

Next, for $J_2$, by using \eqref{ineq:P0CommDecay} and \cref{prop:DL.Oseen.est},
\EQ{
	|J_2|&\lesssim \int_0^t \int_{|x-y|>\frac{|x|} 2} \frac 1 {(|x-y|+\sqrt{t-s})^{4-\be}} \La^{\be} (P_0 \otimes P_0) (x,s)\,dy\,ds
	\\&\lesssim (|x|+1)^{-2-\be }\int_0^t \int_{|x-y|>\frac{|x|} 2} \frac 1 {(|x-y|+\sqrt{t-s})^{4-\be}} \,dy\,ds
	\\&\lesssim_\la (|x|+1)^{-2-\be} |x|^{-1+\be}\lesssim_\la (|x|+1)^{-3}.
}
Lastly, for $J_3$, we need to handle the near- and far-field seperately, and so we break $J_3$ into
\EQ{
	|J_3|&\lesssim \int_0^t \int_{|x-y|>\frac{|x|} 2} \frac 1 {(|x-y|+\sqrt{t-s})^{4-\be}} \La^{\be}( P_0 \otimes P_0) (y,s)\,dy\,ds
	\\&\lesssim_\la \left(\int_{|x-y|>\frac{|x|} 2,|y|>|x|/2}+\int_{|y|\le|x|/2} \right)\frac 1 {|x-y|^{4-\be}} (|y|+1)^{-2-\be}\,dy\\
	&=:J_{31}+J_{32}.}
Then,
\EQ{J_{32}&\lesssim_\la \frac 1 {|x|^{4-\be}} \int_{|y|\le|x|/2} (|y|+1)^{-2-\be}\,dy\,ds
	\\&\lesssim_\la \frac 1 {|x|^{4-\be}}(|x|+1)^{1-\be}\lesssim_\la (|x|+1)^{-3}
} because $\be <1$. Also
\EQ{J_{31} &\lesssim_\la  \int_{|x-y|>\frac{|x|} 2,|y|>|x|/2}\frac 1 {(|x-y|+\sqrt{t-s})^{4-\be}} (|y|+1)^{-2-\be}\,dy \\
	&\lesssim  (|x|+1)^{-2-\be}\int_{|x-y|>\frac{|x|} 2}\frac 1 {|x-y|^{4-\be}}\,dy\,ds\\
	&\lesssim   (|x|+1)^{-2-\be}\frac 1 {|x|^{1-\be}}\lesssim (|x|+1)^{-3}
	.}
Therefore,
\EQ{|\td u|(x,t) \lesssim  (|x|+1)^{-3}.
	}
To extend to all $t>0$ and $|x|>R_0\sqrt t$  as in   \eqref{ineq.HolderDecayCubic}, we appeal to DSS scaling.
\end{proof}
\section{Optimality}\label{sec:optimailty}

In many instances, the estimates developed in this paper hinge critically on bounds for $P_0$. For example, since $\td u = u-P_0$ decays faster than $u$ in Theorem \ref{theorem:Lpdecay1}, the decay for $u$ is optimal precisely when the decay of $P_0$ is optimal. Additionally, in the proofs of Theorems \ref{theorem:tdu.pointwise.decay.CAl} and \ref{theorem:tdu.pointwise.decay.C1Al}, the term $B(P_0,P_0)$ determines the decay of $\td u$ while all other terms constituting $\td u$ decay more rapidly. The decay of $B(P_0,P_0)$ is established following the heuristic 
\[
|B(P_0,P_0)| \sim |P_0\cdot \nb P_0|.
\]
Hence, the sharpness of our results can be checked by developing specific examples at the level of the heat equation. We do not pursue this goal comprehensively, but do provide an example in the simple scenario when $u_0\in L^\I_\loc(\R^3\setminus\{0\})\cap SS$ and is discontinuous. 
In particular, we identify a discontinuous,  SS function $f\in L^\I_\loc(\R^3\setminus \{0\})$ so that
\[
\frac 1 { (|x| +1)}\lesssim 	| \partial_{x_2} e^{\Delta}f | (x,t) ,
\] 
and 
\[
\frac 1 { (|x| +1)}\lesssim 	| e^{\Delta}f | (x,t),
\] 
for $x$ living on the $x_1$-axis.
According to the heuristic above, we thus expect the upper bound in Theorem \ref{theorem:Lpdecay1},
\[
|\td u|(x,t) \sim |B(P_0,P_0)|(x,t) \lesssim \frac {\sqrt t} { (|x| +\sqrt t)^2}, \qquad |x|\geq R_0\sqrt t, 
\]
is optimal.

For simplicity, we do not consider the divergence free property. 
Let $f(x) = \frac {-1} {|x|}$ if $x_2\geq 0$ and $f(x) = \frac 1 {|x|}$ if $x_2\leq 0$. 
Let $x^* = (1,0,0)$ and let $x=\la x^*$.
Consider the caloric extension of $f$ at $t=1$ and $x$:
\EQ{
\partial_{x_2} e^{\Delta} f( x ) &= c \int_{y_2\geq 0}   {-y_2}  e^{ -|x-y|^2/4} \frac {-1} {|y|}\,dy
\\&+ \int_{y_2< 0}   {-y_2}  e^{- |x-y|^2/4} \frac {1} {|y|}\,dy,
}
where we've used the fact that $x_2=0$ for our choice of $x$. 
Since both integrals are positive we have
\EQ{
\partial_{x_2} e^{\Delta} f( x ) &\geq  c \int_{y_2\geq 0}    {y_2}  e^{ -|x-y|^2/4} \frac {1} {|y|}\,dy >0.
}
Since the integrand is non-negative, shrinking the domain of integration decreases the integral and we have 
\EQ{
\partial_{x_2} e^{\Delta} f( x) &\geq c  \int_{ y_2\geq 1; |x-y|<2 }  {y_2}  e^{- |x-y|^2/4} \frac {1} {|y|}\,dy
\\&\geq  \frac C {|x|}\int_{ y_2\geq 1; |x-y|<2 }    e^{- 1}  \,dy,
}
for $|x|\gg 1$ (we emphasize that $x=\la x^*$ so this is not asserted for all points outside of $B_1$). The volume of the region of integration is constant since $x$ sits on the plane $y_2=0$. 
Therefore, for any $\e>0$, we cannot have that
\[
|\nb e^{\Delta}f |(x)\lesssim |x|^{-1-\e}.
\] 
The same proof implies 
\[
|e^{\Delta} f|( x) \geq C |x|^{-1}.
\]

\section*{Acknowledgments} Z.~Bradshaw was supported in part by the Simons Foundation (635438).

\end{document}